\documentclass[12pt,a4paper]{amsart} 
\usepackage[utf8]{inputenc} 
\usepackage[top=3cm, bottom=3.5cm, left=2.5cm, right=2.5cm]{geometry}
\usepackage{amsmath,amsthm,amsfonts,amstext,amssymb}
\usepackage{mathrsfs}
\usepackage{enumerate}
\usepackage{hyperref}
\usepackage{dsfont}
\usepackage{color}
\usepackage{graphicx}
\usepackage{xcolor}
\hypersetup{
    colorlinks,
    linkcolor={blue!50!black},
    citecolor={blue!50!black},
    urlcolor={blue!80!black}
}

\def\R{{\mathbb{R}}}
\def\N{{\mathbb{N}}}
\def\Z{{\mathbb{Z}}}

\newtheorem{theorem}{Theorem}[section]

\newtheorem{lemma}[theorem]{Lemma}

\theoremstyle{definition}

\newtheorem{remark}[theorem]{Remark}

\numberwithin{equation}{section}

\begin{document}

\title[Multi-hop visibility through the vacant set of Poissonian obstacles]{Multi-hop visibility through the vacant set of Poissonian obstacles}
\author{Yingxin Mu}
\address{
  Yingxin Mu,
  University of Leipzig, Institute of Mathematics,
  Augustusplatz 10, 04109 Leipzig, Germany.
}
\email{yingxin.mu@uni-leipzig.de}

\author{Artem Sapozhnikov}
\address{
  Artem Sapozhnikov,
  University of Leipzig, Institute of Mathematics,
  Augustusplatz 10, 04109 Leipzig, Germany.
}
\email{artem.sapozhnikov@math.uni-leipzig.de}

\begin{abstract}
We study multi-hop visibility inside the vacant set of three obstacle models in $\R^d$ with slow decay of spatial correlations and disparate obstacle geometries: Poisson-Boolean models with general i.i.d.\ radii distributions, Poisson cylinders and Brownian interlacements. 
For any $N\geq 0$, we obtain sharp bounds on the probability $P_{\mathrm{vis}}^N(r)$ of visibility to distance $r$ via $(N+1)$ hops in terms of the (explicit) probability of direct visibility to distance $r$ in a given direction, generalizing our earlier result from \cite{MS-Visibility-AIHP} for $N=0$. 
We observe a universal behavior of $P_{\mathrm{vis}}^N(r)$ in terms of two characteristic scales. 
In the three models of interest, these scales are generally the same, but with some anomalous exceptions in low dimensions. 
\end{abstract}

\maketitle

\section{Introduction}

Let $\mathcal C$ be a random closed subset of $\R^d$, viewed as an opaque obstacle. 
We are interested in the visibility through the vacant set of $\mathcal C$, defined as $\mathcal V=\R^d\setminus \mathcal C$. We say that $x$ is \emph{visible} from $y$ if the line segment $[x,y]$ does not intersect $\mathcal C$. In \cite{MS-Visibility-AIHP}, we studied the maximal visibility from the origin in three models of $\mathcal C$ with slow (algebraic) decay of spatial correlations: Poisson-Boolean models, Poisson cylinders and Brownian interlacements---Poisson ensembles of, respectively, balls, doubly infinite cylinders and doubly infinite Wiener sausages in $\R^d$---, and observed the following universal behavior of the probability $P_{\mathrm{vis}}(r)$ that some point at distance $r$ from the origin is visible from the origin,
\begin{equation}\label{eq:visibility-bounds}
c_1\big(\tfrac{r}{\delta_r}\big)^{d-1}f(r)\leq P_{\mathrm{vis}}(r)\leq 
c_2\big(\tfrac{r}{\delta_r}\big)^{d-1}f(r),
\end{equation}
where $f(r)$ is the (explicit) probability that a given point at distance $r$ from the origin is visible from the origin. The scaling function $\delta_r$ is a visibility window on $\partial B(0,r)$: in essence, the conditional probability that $y\in\partial B(0,r)$ is visible from $0$, given that some $x\in\partial B(0,r)$ is, is positive uniformly in $r$ when $\|x-y\|<\delta_r$, and smaller than $\exp(-c\frac{1}{\delta_r}\|x-y\|)$ when $\|x-y\|>\delta_r$. In \cite{MS-Visibility-AIHP}, we have computed $\delta_r$ for the three models: 
\begin{itemize}\itemsep4pt
\item[(a)]
$\delta_r = \frac1r$ for the Poisson-Boolean models in $d\geq 2$, the Poisson cylinders in $d\geq 3$ and the Brownian interlacements in $d\geq 4$, 
\item[(b)]
$\delta_r = \frac{(\log r)^2}{r}$ for the Brownian interlacements in $d=3$, 
\item[(c)]
$\delta_r = 1$ for the Poisson cylinders in $d=2$.
\end{itemize}
Curiously, the scale $\delta_r$ does not depend on the decay rate of spatial correlations (for example, it is the same for all Poisson-Boolean models) and is universal in high dimensions within this class of models, but it does depend on the geometric nature of the obstacle set in low dimensions. The main result of this paper further corroborates this behavior.

\subsection{Main result}
In this paper, we are interested in the multi-hop visibility to a large distance. 
We say that $\partial B(r)$ is visible from the origin via $N$ points if there exist $o_1,\ldots, o_N\in B(r)$ and $o_{N+1}\in\partial B(r)$ such that $o_1$ is visible from the origin and $o_{i+1}$ is visible from $o_i$ for all $1\leq i\leq N$; and we denote the probability of this event by $P_{\mathrm{vis}}^N(r)$. In particular, the direct visibility to distance $r$ corresponds to the special case $N=0$.
\begin{theorem}\label{thm:main}
Let $\mathcal C$ be the occupied set of a Poisson-Boolean model in dimension $d\geq 2$, the Poisson cylinders in dimension $d\geq 3$ or the Brownian interlacements in dimension $d\geq 3$ (see Section~\ref{sec:notation+models} for the precise definitions). Let $\delta_r$ be as above and let $h_r=\sqrt{r\log r}$ for the Brownian interlacements in dimension $d=3$ and, otherwise, $h_r=\sqrt r$.
For any $N\geq 0$, there exist $c = c(d,N,\mathrm{law}(\mathcal C))$ and $C = C(d,N,\mathrm{law}(\mathcal C))$ such that 
\begin{equation}\label{eq:main}
c\,\Big(\frac{r}{\delta_r}\Big)^{d-1}\Big(\frac{h_r}{\delta_r}\Big)^{dN} f(r) 
\leq P_{\mathrm{vis}}^N(r)\leq 
C\,\Big(\frac{r}{\delta_r}\Big)^{d-1}\Big(\frac{h_r}{\delta_r}\Big)^{dN} f(r).
\end{equation}
\end{theorem}
In contrast, if $\mathcal C$ is the occupied set of the Poisson cylinders in $d=2$, then the visibility of $\partial B(r)$ from the origin via $N$ points is equivalent to the direct visibility of $\partial B(r)$ from the origin, since all the vacant components are convex. 
Thus, in this case $P_{\mathrm{vis}}^N(r) = P_{\mathrm{vis}}(r)$, which is comparable to $r^{d-1}f(r)$, as shown in \cite{MS-Visibility-AIHP}.

\smallskip

Theorem~\ref{thm:main} identifies the characteristic scales $\delta_r$ and $h_r$ governing the universal behavior of $P_{\mathrm{vis}}^N(r)$ in three models with rather different correlations structure. These scales do not depend on the decay rate of correlations and are universal in high dimensions, but they do depend on the geometric nature of the obstacle set in low dimensions. 
The function $f(r)$ is explicit in all the three models by \eqref{eq:BM-mu}, \eqref{eq:PC-mu-rho} and \eqref{eq:BI-capacity-rho}.

\subsection{Literature review}
The problem of visibility in random fields of obscuring elements is of interest in various applications, see e.g.\ \cite{Zacks-visibility}. The first mathematical study of visibility goes back to Pólya in 1918. In \cite{Polya-visibility}, he showed that in a forest of trees of constant radius $R$ planted in the vertices of the lattice $\Z^2$, an unobstructed visibility to distance $r$ from the origin is possible if $R=O(\tfrac1r)$ when $r\to\infty$. 
He also considered a random version of the problem, where trees of a fixed radius are planted uniformly at random in a box of size $n$, and computed asymptotics of the probability of visibility to distance $n$ in a given direction as $n\to\infty$. 
In \cite{Calka-visibility}, Calka et al.\ considered visibility through the vacant set of the Poisson-Boolean model in $\R^d$ with balls of fixed radius; they obtained an explicit formula for the probability of visibility in dimension $2$ and gave asymptotic bounds in dimensions $d\geq 3$. Elias and Tykesson \cite{ET-visibility} studied visibility through the vacant set of Brownian interlacements in $\R^d$ and obtained non-sharp bounds on the probability of visibility to distance $r$ in terms of the probability of visibility to distance $r$ in a given direction.
The sharp bounds \eqref{eq:visibility-bounds} were obtained in \cite{MS-Visibility-AIHP} for a class of obstacle models, which contains the Poisson-Boolean models with arbitrary i.i.d.\ random radii, the Poisson cylinders and the Brownian interlacements. In the same setting, it was shown in \cite{MS-Visibility-ECP} that conditioned on $x\in\partial B(r)$ being visible from the origin, the radius of the largest ball centered at $x$, each point of which is visible, rescaled by the visibility window $\delta_r$, converges in distribution to the exponential distribution. Visibility has also been studied in hyperbolic spaces, see e.g.\ \cite{Kahane-90,Kahane-91,BJST-visibility-H,TC-visibility-H}, where it exhibits entirely different behavior than in Euclidean spaces; notably, visibility to infinity has positive probability if the obstacle density is low.

\subsection{Proof ideas}
Let $L = \bigcup_{i=1}^{N+1}[x_{i-1},x_i]$, where $x_0=0$, $x_1,\ldots, x_N\in B(r)$ and $x_{N+1}\in \partial B(r)$. 
The upper bound in \eqref{eq:main} follows from the two key observations: 
\begin{itemize}
\item[(a)]
If at least one of $x_i$'s is at distance bigger than $\kappa h_r$ from the line segment $[0,x_{N+1}]$, then the probability that $x_i$ is visible from $x_{i-1}$ for all $i$ is much smaller than the probability $f(r)$ that $x_{N+1}$ is visible from $x_0$, 
\[
\mathsf P[L\subseteq \mathcal V] \leq e^{-\kappa(\kappa\wedge\log\log r)}f(r),                                                                                                                                                                                     
\]
see Lemma~\ref{l:upper-bound-3}.
\item[(b)]
If all $x_i$'s are within distance $(\kappa+1)h_r$ from the line segment $[0,x_{N+1}]$, then the probability that $x_i$ is visible from $x_{i-1}$ for all $i$ is comparable with the probability that $C_i$ is visible from $C_{i-1}$, where $C_0=\{0\}$, $C_{N+1} = B(x_{N+1},\delta_r)$ and $C_i$ is a cylinder of radius $\delta_r$ and length $\beta_r = r\delta_r/h_r$ centered at $x_i$ and parallel to the line segment $[0,x_{N+1}]$, 
\[
e^{-C(\kappa+1)}\mathsf P\big[\bigcup_{i=1}^{N+1}\text{$C_i$ visible from $C_{i-1}$}\big]\leq 
\mathsf P[L\subseteq \mathcal V]\leq \mathsf P\big[\bigcup_{i=1}^{N+1}\text{$C_i$ visible from $C_{i-1}$}\big],
\]
see Lemmas~\ref{l:upper-bound-1}, \ref{l:upper-bound-2} and \ref{l:visibility-upper-main}. 
\end{itemize}
With (a) and (b), the upper bound on the visibility probability $P_{\mathrm{vis}}^N(r)$ follows by summing over $\kappa\geq 0$ and using the union bound. Indeed, as one can cover $\partial B(r)$ by $O(r^{d-1}/\delta_r^{d-1})$ balls of radius $\delta_r$ and for any $x\in \partial B(r)$, one can cover $B(r)$ by $O(r((\kappa+1)h_r)^{d-1}/(\beta_r\delta_r^{d-1}))$  cylinders of radius $\delta_r$ and length $\beta_r$ parallel and all within distance $(\kappa+1)h_r$ to $[0,x]$, it follows that 
\[
P_{\mathrm{vis}}^N(r) \leq C\,\Big(\frac{r}{\delta_r}\Big)^{d-1}\Big(\frac{r h_r^{d-1}}{\beta_r \delta_r^{d-1}}\Big)^N f(r) =  
C\,\Big(\frac{r}{\delta_r}\Big)^{d-1}\Big(\frac{h_r}{\delta_r}\Big)^{dN} f(r).
\]
For the lower bound in \eqref{eq:main}, we consider a set of tuples $(x_0=0,\ldots, x_{N+1})$ such that all $x_i$'s are within distance $O(h_r)$ from the line segment $[0,x_{N+1}]$ and the projections of all $[x_{i-1},x_i]$'s on the line through $0$ and $x_{N+1}$ do not overlap and have length $O(r)$. We then show that 
\begin{itemize}
\item[(c)]
for any such $(x_0,\ldots, x_{N+1})$, the probability that $x_i$ is visible from $x_{i-1}$ for all $i$ is comparable with the probability $f(r)$ that $x_{N+1}$ is visible from $0$, 
\[
cf(r)\leq \mathsf P[L\subseteq \mathcal V] \leq Cf(r),
\]
see Lemma~\ref{l:lower-bound-1}, and 
\item[(d)]
for any such $(x_0,\ldots, x_{N+1})$ and $(x_0',\ldots, x_{N+1}')$, if $L' = \bigcup_{i=1}^{N+1}[x_{i-1}',x_i']$, then 
\[
\mathsf P[L'\subseteq\mathcal V\,|\,L\subseteq\mathcal V] \leq C\exp\Big(-\frac{c}{\delta_r}\min\big(d_H(L,L'),1\big)\Big),
\]
see Lemma~\ref{l:lower-bound-2}, 
where $d_H$ is the Hausdorff distance between $L$ and $L'$ (in fact, we use a compatible to the Hausdorff metric, which is more convenient in our setting). 
\end{itemize}
With (c) and (d), the lower bound on the visibility probability $P_{\mathrm{vis}}^N(r)$ follows from the second moment argument and the observation that if $d_H(L,L')\leq k\delta_r$, then each $x_i'$ lies in a cylinder of radius $k\delta_r$ and length $k\beta_r$ centered at $x_i$ and parallel to the line through $0$ and $x_{N+1}$, and $x_{N+1}'$ lies in the ball of radius $k\delta_r$ around $x_{N+1}$. 

\smallskip

The above considerations are general and apply to any obstacle set $\mathcal C$ with rotational invariant law: In order to prove the bounds for a specific model, it suffices to verify the properties (a)--(d) for suitable scales $\delta_r$ and $h_r$. 

\smallskip

For the three models of interest, we computed $\delta_r$ in \cite{MS-Visibility-AIHP}. The reasoning for the choice of $h_r$ is the following. 
The probability $\mathsf P[L\subseteq \mathcal V]$ is expressed in terms of (1) the volume of the $\rho$-neighborhood $B(L,\rho)$ of $L$ in the case of the Poisson-Boolean model, (2) the $(d-1)$-volume of the projection of $B(L,\rho)$ on a random hyperplane in the case of the Poisson cylinders, and (3) the capacity of $B(L,\rho)$ in the case of the Brownian interlacements. Thus, the remaining work is to compare the volumes resp.\ capacities of $\rho$-neighborhoods of $L$ to those of $[0,x_{N+1}]$ or $L'$. If $L$ is as in (c), then the length of $L$ satisfies 
\[
r\leq \sum_{i=1}^{N+1}\|x_{i-1}-x_i\| \leq r + C\frac{h_r^2}{r}.
\]
If $h_r^2/r = 1$, then the volume of $B(L,\rho)$ is comparable to the volume of the $\rho$-neighborhood of $[0,x_{N+1}]$, and if $h_r^2/r = \mathds{1}_{d\geq4} + \log r\mathds{1}_{d=3}$, then the capacity of $B(L,\rho)$ is comparable to the capacity of the $\rho$-neighborhood of $[0,x_{N+1}]$ (see \eqref{eq:capacity-asymptotics} and Lemma~\ref{l:capacity-cylinder-d3}). This observation determines the choice of $h_r$. (See also Remark~\ref{rem:hr-in-d3} about the subtle relevance of the choice of $h_r$ in the proof of the upper bound for the Brownian interlacements.)

\smallskip

While the volume of cylindrical sets can be explicitly computed, which allows for an explicit comparison of volumes of different cylindrical sets, the capacity of cylindrical sets is not explicit. This makes the comparison of capacities of various cylindrical sets substantially more difficult, see Sections~\ref{sec:upper-bound-BI-2} and \ref{sec:lower-bound-BI}, which present the main technical contributions of our work.

\subsection{Structure of the paper} In Section~\ref{sec:notation+models}, we introduce the common notation used throughout the paper and define the three models of interest precisely. 
In Section~\ref{sec:auxiliary-results}, we provide preliminary results on the Brownian motion, capacity and equilibrium measure, which will be useful in the proof of the main result for the Brownian interlacements, as well as some basic geometric properties, which will be useful in the context of Poisson-Boolean models and Poisson cylinders. We prove the upper bound of \eqref{eq:main} in Section~\ref{sec:upper-bound}; we first reduce the statement to several general properties of the obstacle set (Lemmas~\ref{l:upper-bound-1}--\ref{l:upper-bound-3}) and then prove these properties for the Poisson-Boolean models in Section~\ref{sec:upper-bound-BM}, for the Poisson cylinders in Section~\ref{sec:upper-bound-PC}, and for the Brownian interlacements in Section~\ref{sec:upper-bound-BI-1} and \ref{sec:upper-bound-BI-2}. The lower bound of \eqref{eq:main} is proven in Section~\ref{sec:lower-bound}; here, we also first reduce the statement to some general properties of the obstacle set (Lemmas~\ref{l:lower-bound-1} and \ref{l:lower-bound-2}) and then verify them separately for each of the three models, for the Poisson-Boolean models in Section~\ref{sec:lower-bound-BM}, for the Poisson cylinders in Section~\ref{sec:lower-bound-PC}, and for the Brownian interlacements in Section~\ref{sec:lower-bound-BI}.

\section{Notation and models}\label{sec:notation+models}

In this section, we introduce some common notation, used throughout the paper, and describe the three models of interest precisely.

\smallskip

Let $x\in\R^d$, $\rho>0$ and $K,K'\subset \R^d$. We denote by $B(x,\rho)$ the closed Euclidean ball at $x$ with radius $\rho$ and and by $\partial B(x,\rho)$ its boundary; we write $B(\rho)=B(0,\rho)$ and $\partial B(\rho) = \partial B(0,\rho)$. We denote the closed $\rho$-neighborhood of $K$ by $B(K,\rho$), that is $B(K,\rho) = \bigcup_{x\in K}B(x,\rho)$. 
The Euclidean distance between $K$ and $K'$ is denoted by $d(K,K')$ and we write $d(x,K)$ instead of $d(\{x\},K)$ for the distance from $x$ to $K$. 
The $k$-dimensional Lebesgue measure is denoted by $\lambda_k$, and we denote by $\kappa_k$ the $\lambda_k$-volume of the $k$-dimensional unit ball, that is $\kappa_k = \pi^{k/2}/\Gamma(k/2+1)$. 

We write $e_1,\ldots, e_d$ for the canonical orthonormal basis of $\R^d$. For $x\in \R^d$, we write $x(i)$ for the $i$-th coordinate of $x$ in the canonical basis. 

For $x\neq y\in \R^d$, we write $\ell_{x,y}^\infty$ for the unique infinite line passing through $x$ and $y$. For a line segment $L = [x,y]$, we write $|L|$ for the length $\|x-y\|$ of $L$.  

\smallskip

Let $\mathcal C\subseteq \R^d$ be an obstacle set. We write $\mathcal V = \R^d\setminus \mathcal C$ for the vacant set of $\mathcal C$. 
For sets $A_1,A_2$, we say that $A_1$ is visible from $A_2$ (through the vacant set of $\mathcal C$) if there exist $x_1\in A_1$ and $x_2\in A_2$ such that the line segment $[x_1,x_2]$ is contained in $\mathcal V$. 
For sets $A_1,\ldots, A_k$, we write $\mathrm{vis}(A_1,\ldots, A_k)$ for the event that $A_{i+1}$ is visible from $A_i$ for each $1\leq i<k$; if $A_i=\{a_i\}$ then we omit the parentheses from the notation. 

\smallskip

We define the scaling functions $\delta_r$, $h_r$ and $\beta_r$ as follows:
\begin{itemize}
\item[(a)]
For the Poisson-Boolean models in dimensions $d\geq 2$, the Poisson cylinder model in dimensions $d\geq 3$, and the Brownian interlacements in dimensions $d\geq 4$, 
\begin{equation}\label{eq:delta-beta-h-1}
\delta_r = \frac1r,\,\,h_r = \sqrt{r},\,\,\beta_r = \frac{r\delta_r}{h_r} = \frac{1}{\sqrt{r}},
\end{equation}
\item[(b)]
for the Brownian interlacements in dimension $d=3$, 
\begin{equation}\label{eq:delta-beta-h-2}
\delta_r = \frac{(\log r)^2}r,\,\,h_r =\sqrt{r\log r},\,\,\beta_r = \frac{r\delta_r}{h_r} =\frac{(\log r)^{3/2}}{\sqrt{r}}.
\end{equation}
\end{itemize}

Throughout the paper, we use $c$ and $C$ to denote strictly positive constants that depend only on $d$, $N$ and the distribution of the obstacle set $\mathcal C$. The exact values of these constants are allowed to change at each occurrence, even within the same string of inequalities.

\subsection{Poisson-Boolean models}

Let $d\geq 2$. For $\alpha>0$ and a probability distribution $\mathsf Q$ on $\R_+$, we consider a Poisson point process $\omega = \sum_{i\geq 1}\delta_{(x_i,r_i)}$ on $\R^d\times\R_+$ with intensity measure $\alpha dx\otimes \mathsf Q$. The support of $\omega$ induces the closed random subset $\mathcal C$ of $\R^d$, 
\[
\mathcal C = \mathcal C(\omega) =\bigcup\limits_{i\geq 1}B(x_i,r_i),
\]
which is called the \emph{Poisson-Boolean model with intensity $\alpha$ and radii distribution $\mathsf Q$}; we denote its distribution by $\mathsf P^\alpha_{\mathsf Q}$. The set $\mathcal C$ does not coincide with $\R^d$ if and only if 
\[
\mathsf E[\rho^d]<\infty, 
\]
see e.g.\ \cite[Proposition~3.1]{MR-Book}, which we will always assume in the paper. In this and subsequent discussions of the Poisson-Boolean models, we denote by $\rho$ a generic random variable with distribution $\mathsf Q$. 

The number of balls that intersect a compact set $K$ is a Poisson random variable with parameter 
\[
\int\limits_{(x,r)\,:\,B(x,r)\cap K\neq\emptyset}\alpha dx\otimes \mathsf Q(dr) = \alpha\int\limits_{\R^d}\mathsf P\big[\rho\geq d(x,K)\big]\,dx
= \alpha\,\mathsf E\big[\lambda_d\big(B(K,\rho)\big)\big].
\]
Thus, 
\begin{equation}\label{eq:BM-mu}
\mathsf P^\alpha_\mathsf Q \big[\mathcal C\cap K = \emptyset\big] = \exp\Big(-\alpha\mathsf E\big[\lambda_d\big(B(K,\rho)\big)\big]\Big),\quad\text{for compact }K\subset\R^d. 
\end{equation}
The relations \eqref{eq:BM-mu} characterize the law $\mathsf P^\alpha_\mathsf Q$ of $\mathcal C$.\footnote{The law of $\mathcal C$ is a probability measure on $(\Sigma,\mathscr F)$, where $\Sigma$ is the set of closed sets in $\R^d$ and $\mathscr F$ is the sigma-algebra on $\Sigma$ generated by sets $\{F\in\Sigma\,:\,F\cap K=\emptyset\}$ for compacts $K$, see e.g.\ \cite[Chapter~2]{Matheron}.}
We refer to \cite{MR-Book} for an introduction to Boolean models and to \cite{Gouere08,ATT18,Pen18,DRT-Boolean} and references therein for futher results.

\subsection{Poisson cylinders}

Let $d\geq 2$. Let $\mathbb L$ be the set of lines in $\R^d$ 
defined as the affine Grassmanian of $1$-dimensional affine subspaces of $\R^d$. Let $SO_d$ be the topological group of rigid rotations of $\R^d$ and denote by $\nu$ the unique Haar measure on $SO_d$ such that $\nu(SO_d)=1$. Let $\ell$ be the line along the vector $e_1$ and $H$ the hyperplane orthogonal to $e_1$. Let $\gamma:H\times SO_d\to \mathbb L$ be defined as $\gamma(x,\phi) = \phi(x+\ell)$, where $x+\ell$ is the translation of $\ell$ by $x$, and consider the measure 
\[
\mu(\cdot) = (\lambda_{d-1}\otimes\nu)\big(\gamma^{-1}(\cdot)\big)
\]
on $\mathbb L$. For $\alpha>0$, we consider a Poisson point process $\omega = \sum_{i\geq 1}\delta_{L_i}$ on $\mathbb L$ with intensity measure $\alpha\mu$, which is called the \emph{Poisson lines process at level $\alpha$}.  
For each $\rho>0$, the support of $\omega$ induces the closed random subset $\mathcal C$ of $\R^d$, 
\[
\mathcal C = \mathcal C(\omega) =\bigcup\limits_{i\geq 1}B(L_i,\rho),
\]
which is called the \emph{Poisson cylinders process at level $\alpha$ with radius $\rho$}; we denote its distribution by $\mathsf P^\alpha_\rho$. The law of $\mathcal C$ is translation and rotation invariant and exhibits long-range correlations, 
\[
c\big(1 +\|x-y\|\big)^{1-d}\leq \mathrm{cov}^\alpha_\rho\big(\mathds{1}_{x\in\mathcal C},\mathds{1}_{y\in\mathcal C}\big)\leq C\big(1 +\|x-y\|\big)^{1-d},
\]
see \cite{TW-cylinders}. 
For $K\subset\R^d$ and $\phi\in SO_d$, we denote by $K_\phi$ the rotation of $K$ by $\phi$. Let $\pi$ be the orthogonal projection on $H$. By the definition of the intensity measure $\mu$, 
\begin{equation}\label{eq:PC-mu-rho}
\mathsf P^\alpha_\rho \big[\mathcal C\cap K = \emptyset\big] = \exp\Big(-\alpha\,\int\limits_{SO_d}\lambda_{d-1}\big(\pi(B(K,\rho)_\phi)\big)\,\nu(d\phi)\Big),\quad\text{for compact }K\subset\R^d. 
\end{equation}
The relations \eqref{eq:PC-mu-rho} characterize the law $\mathsf P^\alpha_\rho$ of $\mathcal C$. We refer the reader to \cite[Chapter~13]{SchneiderWeil} and \cite{TW-cylinders} for further details of the Poisson lines and cylinders models and to \cite{TW-cylinders, HST-PC-d3, BT-PC-connected, AT-PC-decoupling} and references therein for further results.

\subsection{Brownian interlacements}

Let $d\geq 3$.
Let $\mathsf W_+$ be the space of continuous $\R^d$-valued paths tending to infinity at infinite times. We denote by $X_t$, $t\geq 0$, the canonical process on $\mathsf W_+$  (i.e.\ $X_t(w) = w(t)$) and by $\mathcal W_+$ the sigma-algebra on $\mathsf W_+$ generated by the canonical process. 
Since the Brownian motion is transient in dimension $d\geq 3$, the law of the Brownian motion starting from $x\in\R^d$, denoted by $\mathsf P_x$, is a probability measure on $(\mathsf W_+,\mathcal W_+)$.
Given a compact set $K$ in $\R^d$, we denote by $e_K$ the \emph{equilibrium measure} of $K$, which is a finite measure uniquely determined by the last visit formula (see \cite[Theorem 3.4]{Sznitman-BMbook}):
\[
\mathsf P_x\big[L_K>0,\,L_K\in dt,\,X_{L_K}\in dy\big] = p_t(x,y)\,e_K(dy)\,dt,
\]
where $L_K(w) = \sup\{t>0\,:\,X_t(w)\in K\}$, $w\in \mathsf W_+$,
is the time of the last visit of $w$ to $K$, and $p_t(\cdot,\cdot)$ is the Brownian transition density. The total mass of $e_K$ is called the \emph{capacity} of $K$ and is denoted by $\mathrm{cap}(K)$.
For a closed ball $B$ and $x\notin \mathrm{int}(B)$, we denote by $\mathsf P^B_x$ the law of the Brownian motion starting from $x$ and conditioned on staying outside of $B$ for all positive times (see e.g.\ \cite[Theorems 4.1 and 2.2.]{Burdzy-87}).

\smallskip

Let $\mathsf W$ be the space of continuous doubly-infinite $\R^d$-valued paths tending to infinity at positive and negative infinite times. We denote by $X_t$, $t\in\R$, the canonical process on $\mathsf W$  (i.e.\ $X_t(w) = w(t)$) and by $\mathcal W$ the sigma-algebra on $\mathsf W$ generated by the canonical process. We denote the canonical time shift on $\mathsf W$ by $\theta_t$, $t\in\R$. For a closed set $F\subseteq \R^d$, we define the first entrance time of $w\in\mathsf W$ in $F$ as 
$H_F(w)= \inf\{t\in\R\,:\,X_t(w)\in F\}$.
For any closed ball $B$ of positive radius, we define the measure $Q_B$ on $(\mathsf W,\mathcal W)$ by
\[
Q_B\big[(X_{-t})_{t\geq 0}\in A,\,X_0\in dx,\,(X_t)_{t\geq 0}\in A'\big] = e_B(dx)\,\mathsf P^B_x[A]\,\mathsf P_x[A'],\quad A,A'\in \mathcal W_+.
\]
By \cite[Lemma 2.1]{Sznitman-BI}, the measures $Q_B$ are compatible, so that for any compact set $K$, the measure 
\[
Q_K[\cdot] = \mathds{1}_{\{H_K<\infty\}}\,Q_B[\theta_{H_K}^{-1}(\cdot)],
\]
on $\mathsf W$, where $B$ is any closed ball containing $K$, is well-defined. Moreover, $Q_K$ is a finite measure with $Q_K[\mathsf W] = \mathrm{cap}(K)$. 

\smallskip

Two paths $w$ and $w'$ in $\mathsf W$ are called equivalent, if $w'=\theta_t(w)$ for some $t\in\R$. The quotient set of $\mathsf W$ modulo this equivalence relation is denoted by $\mathsf W^*$. The canonical projection $\pi^*:\mathsf W\to \mathsf W^*$ induces the sigma-algebra $\mathcal W^* = \{A\subseteq \mathsf W^*\,:\,(\pi^*)^{-1}(A)\in \mathcal W\}$ on $\mathsf W^*$. 
By \cite[Theorem 2.2]{Sznitman-BI}, there exists a unique sigma-finite measure $\nu$ on $(\mathsf W^*,\mathcal W^*)$, whose restriction to any $\mathsf W_K^* = \pi^*(\{w\,:\,H_K(w)<\infty\})$ coincides with $Q_K$, more precisely, 
\[
\mathds{1}_{\mathsf W_K^*}\,\nu = \pi^*\circ Q_K,\quad \text{for any compact $K$ in $\R^d$}.
\]
Note that $\nu[\mathsf W^*_K] = Q_K[\mathsf W] = \mathrm{cap}(K)$. 

\smallskip

For $\alpha>0$, we consider a Poisson point process $\omega = \sum_{i\geq 1}\delta_{w_i^*}$  on $\mathsf W^*$ with intensity measure $\alpha\nu$, which is called the \emph{Brownian interlacement point process at level $\alpha$}. For each $\rho>0$, the support of $\omega$ induces the closed random subset $\mathcal C$ of $\R^d$, 
\[
\mathcal C = \mathcal C(\omega) = \bigcup_{i\geq1}B\big(\mathrm{range}(w_i^*),\rho\big),
\]
which is called the \emph{Brownian interlacements at level $\alpha$ with radius $\rho$}; we denote its distribution by $\mathsf P^\alpha_\rho$. Here, $\mathrm{range}(w_i^*) = \bigcup_{t\in\R}\{w_i(t)\}$ for any path $w_i$ in $\pi^{-1}(w_i^*)$. 
The law of $\mathcal C$ is translation and rotation invariant and exhibits long-range correlations, 
\[
c\big(1 +\|x-y\|\big)^{2-d}\leq \mathrm{cov}^\alpha_\rho\big(\mathds{1}_{x\in\mathcal C},\mathds{1}_{y\in\mathcal C}\big)\leq C\big(1 +\|x-y\|\big)^{2-d},
\]
see \cite{Li-BI}. 
By the definition of the intensity measure, 
\begin{equation}\label{eq:BI-capacity-rho}
\mathsf P^\alpha_\rho \big[\mathcal C\cap K = \emptyset\big] =
\exp\Big(- \alpha \mathrm{cap}(B(K,\rho))\Big),\quad\text{for compact }K\subset\R^d,
\end{equation}
see \cite[(2.32)]{Sznitman-BI}. The relations \eqref{eq:BI-capacity-rho} characterize the law $\mathsf P^\alpha_\rho$ of $\mathcal C$. We refer the reader to \cite{Sznitman-BI} for further details on the construction of the Brownian interlacements and to \cite{Sznitman-BI, Li-BI,MS-BI} and references therein for further results.

\section{Auxiliary results}\label{sec:auxiliary-results}

\subsection{Brownian motion and capacity}
In this section, we collect some properties of the Brownian motion, the capacity and the equilibrium measure of cylindrical sets, which will be used in the proof of the main result for the Brownian interlacements. 

Recall from the previous section that $\mathsf P_x$ denotes the law of the Brownian motion started in $x$ and $H_K$ is the first hitting time of $K$.  
Classically, for any $R_1<R_2$ and $y\in\R^d$ with $R_1<\|y\|<R_2$, 
\begin{equation}\label{eq:BM-hitting}
\mathsf P_y\big[H_{\partial B(R_2)}<H_{\partial B(R_1)}\big] = 
\left\{\begin{array}{ccl}
\frac{\log R_1 - \log \|y\|}{\log R_1 - \log R_2} &\quad& d=2\\[10pt]
\frac{R_1^{2-d}-\|y\|^{2-d}}{R_1^{2-d}-R_2^{2-d}} &\quad& d\geq 3
\end{array}\right.
\end{equation}
(see e.g.\ \cite[Theorem~3.18]{MP-BM-book}); in particular, when $d\geq 3$, 
\begin{equation}\label{eq:BM-escape}
\mathsf P_y\big[H_{\partial B(R_1)}=\infty\big] = 1 - \frac{\|y\|^{2-d}}{R_1^{2-d}}\,.
\end{equation}
Let $d\geq 3$. 
Let $\sigma_R$ be the uniform distribution on $\partial B(R)$ and define $\mu_R=\tfrac{2\pi^{d/2}R^{d-2}}{\Gamma(d/2-1)}\sigma_R$. 
For a sigma-finite measure $\mu$, we write $\mathsf P_\mu$ for $\int_{\R^d}\mathsf P_x[\cdot]\mu(dx)$. 
For any compact set $K$ in $\R^d$ and any $R$ such that $K\subset B(R)$, the equilibrium measure $\mu_K$ and the capacity $\mathrm{cap}(K)$ of $K$, introduced in the previous section, are equivalently defined as 
(see \cite[Theorem~3.1.10]{PortStone})
\begin{equation}\label{def:equilibrium-measure-and-capacity}
\mu_K = \mathsf P_{\mu_R}\big[H_K<\infty, X_{H_K}\in\cdot\big]\quad\text{resp.}\quad
\mathrm{cap}(K) = \mu_K(K) = \mathsf P_{\mu_R}\big[H_K<\infty\big].
\end{equation}
Capacity is an invariant under isometries and monotone function on compacts, $\mathrm{cap}(\rho K) = \rho^{d-2}\mathrm{cap}(K)$ and $\mathrm{cap}(K) = \mathrm{cap}(\partial K)$ (see \cite[Proposition~3.1.11]{PortStone}). 
Furthermore (see e.g.\ \cite[Theorem~2.4.9]{Sznitman-BMbook}), 
\begin{equation}\label{eq:capacity-variational}
\frac{1}{\mathrm{cap}(K)} = \inf\Big\{\iint G(x,y)\mu(dx)\mu(dy)\,:\,\text{$\mu$ is a probability measure on $K$}\Big\},
\end{equation}
where $G(x,y) = \frac{\Gamma(\frac{d-2}2)}{2\pi^{d/2}}\|x-y\|^{2-d}$ is the Green function for the standard Brownian motion, and the infimum is attained when $\mu$ is the normalized equilibrium measure of $K$. 

By \cite[Theorem~3.2.1]{PortStone}, the following relation between the hitting probability of a set and its equilibrium measure holds: for all compact $K$ and $x\notin K$, 
\begin{equation}\label{eq:hitting-probability-capacity}
\mathsf P_x[H_K<\infty] = \int G(x,y)\mu_K(dy).
\end{equation}

\smallskip

In the remainder of this section, we collect some results about the capacity and the equilibrium measure of cylindrical sets. 
By \cite[Lemma~2.1]{MS-Visibility-ECP}, 
\begin{equation}\label{eq:capacity-asymptotics}
\mathrm{cap}\big(B([0,x],\rho)\big) = \varkappa_d\rho^{d-3}\Big(\tfrac{\|x\|}{\log \|x\|}\mathds{1}_{d=3} + \|x\|\mathds{1}_{d\geq 4}\Big)(1+o(1)),\quad\text{as }x\to\infty,
\end{equation}
where $\varkappa_3= \pi$ and $\varkappa_d =2\pi^{\frac{d-1}2}/\Gamma(\frac{d-3}{2})$ for $d\geq 4$.
\begin{lemma}\label{l:equilibrium-measure-cylinder}
Let $r\geq 1$. Let $\mu$ be the normalized equilibrium measure of $B([0,re_1],\rho)$. There exist $c=c(d,\rho)>0$ and $C=C(d,\rho)<\infty$ such that for any $0\leq a<b\leq r$, 
\[
c\frac{b-a}{r}\leq \mu\big(B([ae_1,be_1],\rho)\big)\leq C\frac{b-a}{r}.
\]
\end{lemma}
\begin{proof}
Let $H_{a,b}$ be the first entrance time of the Brownian motion in $B([ae_1,be_1],\rho)$. Let $N=\lceil r/(b-a)\rceil$ and define $x_n=r+(b-a)n$, for $0\leq n\leq N$. 
By the definitions of the capacity and the equilibrium measure in \eqref{def:equilibrium-measure-and-capacity}, for all $R$ large enough, 
\begin{multline*}
\mathrm{cap}\big(B([0,3re_1],\rho)\big) = \mathsf P_{\mu_R}[H_{0,3r}<\infty]\\
\begin{aligned}
&\leq \mathsf P_{\mu_R}[H_{0,r}<\infty] +  \mathsf P_{\mu_R}[H_{2r,3r}<\infty] + \sum_{n=1}^N\mathsf P_{\mu_R}[H_{x_{n-1},x_n}=H_{0,3r}<\infty]\\
&\leq 2\,\mathrm{cap}\big(B([0,re_1],\rho)\big) + N\mathsf P_{\mu_R}[H_{a,b}=H_{0,r}<\infty]\\
&= 2\,\mathrm{cap}\big(B([0,re_1],\rho)\big) + N\,\mathrm{cap}\big(B([0,re_1],\rho)\big)\,\mu\big(B([ae_1,be_1],\rho)\big).
\end{aligned}
\end{multline*}
By \eqref{eq:capacity-asymptotics}, we obtain that $\mu\big(B([ae_1,be_1],\rho)\big)\geq c\frac{b-a}{r}$. 

\smallskip

For the upper bound, it suffices to prove it for $b-a<\frac14r$. Let $M=\lfloor r/4(b-a)\rfloor$ and define
$y_m = (b-a)m$, for $0\leq m\leq M$. 
By the definitions of the capacity and the equilibrium measure in \eqref{def:equilibrium-measure-and-capacity}, for all $R$ large enough, 
\begin{eqnarray*}
\mathrm{cap}\big(B([0,\tfrac14re_1],\rho)\big) 
&= &\mathsf P_{\mu_R}[H_{0,r/4}<\infty]
\geq \sum_{i=1}^M P_{\mu_R}[H_{y_{i-1},y_i}=H_{0,r/4}<\infty]\\
&\geq &M\mathsf P_{\mu_R}[H_{a,b}=H_{0,r}<\infty]\\ 
&= &M\,\mathrm{cap}\big(B([0,re_1],\rho)\big)\,\mu\big(B([ae_1,be_1],\rho)\big).
\end{eqnarray*}
By \eqref{eq:capacity-asymptotics}, we obtain that $\mu\big(B([ae_1,be_1],\rho)\big)\leq C\frac{b-a}{r}$. 
\end{proof}

\begin{lemma}\label{l:capacity-cylinder-d3}
Let $d=3$. 
Let $\rho,D>0$. There exists $C=C(\rho,D)$ such that for all $r\geq 1$, 
\[
\mathrm{cap}\big(B([0,(r+D\log r)e_1],\rho)\big) \leq 
\mathrm{cap}\big(B([0,re_1],\rho)\big) + C.
\]
\end{lemma}
\begin{proof}
By the scaling property and the monotonicity of the capacity, 
\begin{eqnarray*}
\mathrm{cap}\big(B([0,(r+D\log r)e_1],\rho)\big) 
&= &\Big(\frac{r+D\log r}{r}\Big)^{d-2}\,
\mathrm{cap}\Big(B\big([0,re_1],\frac{r}{r+D\log r}\rho\big)\Big)\\
&\leq 
&\Big(\frac{r+D\log r}{r}\Big)^{d-2}\,
\mathrm{cap}\big(B([0,re_1],\rho)\big)\\
&\leq &\mathrm{cap}\big(B([0,re_1],\rho)\big) + C,
\end{eqnarray*}
where the last inequality follows from \eqref{eq:capacity-asymptotics}. 
\end{proof}

\subsection{Geometric preliminaries}
In this section, we collect some auxiliary geometric properties used in the proof of the main result for the Poisson-Boolean and Poisson cylinders models. 

\begin{lemma}\label{l:volume-cylinder}
Let $d\geq 2$. For any $x\in \R^d$ and $\rho>0$,
\[
\lambda_d\big(B([0,x],\rho)\big) = \kappa_{d-1}\rho^{d-1}\|x\| + \kappa_d\rho^d,
\]
where $\kappa_s$ is the volume of the $s$-dimensional Euclidean unit ball.
\end{lemma}
\begin{proof}
Immediate from the definition of $B([0,x],\rho)$. 
\end{proof}

\begin{lemma}\label{l:intersection-two-lines}
Let $\ell_1$ and $\ell_2$ be lines in $\R^d$ intersecting in a single point $z$, and let $\theta\in(0,\frac\pi2]$ be the angle between them. 
Let $u_1$ be a unit vector parallel to $\ell_1$. 
Then for any $\delta>0$,  
\[
B(\ell_1,\delta)\cap B(\ell_2,\delta)\subseteq 
z + \Big\{x\in B(\ell_1,\delta)\,:\,\big|\langle x,u_1\rangle\big|\leq \frac{\delta}{\tan(\theta/2)}\Big\}.
\]
In words, the intersection of $\delta$-neighborhoods of the two lines is contained in the finite cylinder centered at $z$ of radius $\delta$ and length $2\frac{\delta}{\tan(\theta/2)}$ whose axis is aligned with $\ell_1$. 
\end{lemma}
\begin{proof}
Let $\Pi$ be the plane spanned by $\ell_1$ and $\ell_2$. The intersection of the $\delta$-neighborhoods of $\ell_1$ and $\ell_2$ in $\Pi$ is a rhombus with internal angles $\theta$ and $\pi-\theta$. 
The angle between $\ell_1$ and the line connecting the farthest vertices of the rhombus equals $\theta/2$. Since any fatherst point of the rhombus is at distance $\delta$ from $\ell_1$, its projection on $\ell_1$ is at distance $\frac{\delta}{\tan(\theta/2)}$ from $z$. The result follows. 
\end{proof}

\begin{lemma}\label{l:volume-difference}
Let $\rho>0$. For any line $\ell$ and $y,y'\in\R^d$ with $d(y',\ell)\geq d(y,\ell)$,
\[
\lambda_d\big(B([y,y'],\rho)\setminus B(\ell,\rho)\big)\geq 
c\rho^{d-2}\min(\rho,1)\, \|y-y'\| \min\big(d(y',\ell),1\big),
\]
for some $c=c(d)>0$. 
\end{lemma}
\begin{proof}
Without loss of generality, we may assume that $d([y,y'],\ell)\geq \frac13 d(y',\ell)$. Otherwise, we divide $[y,y']$ into three segments $[y,y_1]$, $[y_1,y_2]$ and $[y_2,y']$ of equal length and note that either $d([y,y_1],\ell)\geq \frac13 d(y',\ell)\geq\frac13 \max(d(y,\ell),d(y_1,\ell))$ or $d([y_2,y'],\ell)\geq \frac13 d(y',\ell)=\frac13\max(d(y_2,\ell),d(y',\ell))$. 

If $d(y',\ell)>3\rho$ then $d([\frac12(y+y'),y'],\ell)>2\rho$. Hence by Lemma~\ref{l:volume-cylinder}, 
\begin{eqnarray*}
\lambda_d\big(B([y,y'],\rho)\setminus B(\ell,\rho)\big)
&\geq 
&\lambda_d\big(B([\tfrac12(y+y'),y'],\rho)\geq \frac12\kappa_{d-1}\rho^{d-1}\|y-y'\|\\
&\geq &\frac12\kappa_{d-1}\rho^{d-1} \|y-y'\| \min\big(d(y',\ell),1\big).
\end{eqnarray*}
If $d(y',\ell)\leq 3\rho$ and $\|y-y'\|\leq 3\rho$, then, by Lemma~\ref{l:volume-difference-balls}, 
\begin{eqnarray*}
\lambda_d\big(B([y,y'],\rho)\setminus B(\ell,\rho)\big)
&\geq &\lambda_d\big(B(y',\rho)\setminus B(\ell,\rho)\big)
\geq 
\frac12\kappa_d\rho^{d-1}\min\big(d(y',\ell),2\rho\big)\\
&\geq &\frac13\kappa_d\rho^{d-1} \|y-y'\| \min\big(d(y',\ell),1\big).
\end{eqnarray*}

Finally, if $d(y',\ell)\leq 3\rho$ and $\|y-y'\|\geq 3\rho$, then the projection of $[y,y']$ on $\ell$ has length at least $\frac12\|y-y'\|$. 
Let $H_z$ be the hyperplane through $z$ orthogonal to $\ell$. Then, by Lemma~\ref{l:volume-difference-balls},
\begin{eqnarray*}
\lambda_d\big(B([y,y'],\rho)\setminus B(\ell,\rho)\big)
&\geq &\int_{[y,y']}\lambda_{d-1}\big(H_z\cap (B(z,\rho)\setminus B(\ell,\rho))\big)dz\\
&\geq &\big(\frac12\|y-y'\|\big)\big(\frac12\kappa_{d-1} \rho^{d-2} \min\big(\tfrac13 d(y',\ell),2\rho\big)\big)\\
&\geq &\frac{1}{12}\kappa_{d-1}\rho^{d-2} \|y-y'\| \min\big(d(y',\ell),1\big),
\end{eqnarray*}
as required.
\end{proof}
\begin{lemma}\label{l:volume-difference-balls}
Let $\rho>0$. For any $x,y\in \R^d$, 
\[
\lambda_d\big(B(x,\rho)\setminus B(y,\rho)\big)\geq \frac12\kappa_d \rho^{d-1} \min\big(\|x-y\|,2\rho\big).
\]
\end{lemma}
\begin{proof}
We may assume that $x=(0,\ldots,0)$ and $y=(s,0,\ldots, 0)$ with $s>0$. If $s>2\rho$, then the balls do not overlap and the result is trivial. Let $s\in[0,2\rho]$. 

The volume of a spherical cap of height $h\in[0,\rho]$ and radius $\rho$ equals 
\[
V(\rho,h) = \kappa_{d-1}\int_{\rho-h}^\rho(\rho^2 - t^2)^{\frac{d-1}{2}}dt = \kappa_{d-1}\int_0^h (s(2\rho-s)^{\frac{d-1}{2}}ds.
\]
Since $V$ is a convex function in $h$, 
\[
\frac{\partial V}{\partial h} = \kappa_{d-1} \big(h(2\rho-h)\big)^{\frac{d-1}{2}},\quad 
\frac{\partial^2 V}{\partial h^2} = \frac{d-1}2 \big(h(2\rho-h)\big)^{\frac{d-3}{2}} 2(\rho-h)\geq 0,
\]
we have 
\[
V(\rho,h)\leq \frac{h}{\rho}V(\rho,\rho) = \frac12 h\rho^{d-1}\kappa_d.
\]
It remains to notice that for $h=\rho - \frac12\|x-y\|$, 
\[
\lambda_d\big(B(x,\rho)\setminus B(y,\rho)\big) = \kappa_d\rho^d - 2 V(\rho,h) \geq \kappa_d(\rho-h)\rho^{d-1} = \frac12\kappa_d\rho^{d-1}\|x-y\|,
\]
as required.
\end{proof}

\begin{lemma}\cite[Lemma~4.3]{MS-Visibility-AIHP}\label{l:vector-projections}
Let $d\geq 3$. Let $H_s$ be the hyperplane $\{z\in\R^d\,:\,z(s)=0\}$ and $\pi_s$ the orthogonal projection on $H_s$. For any vectors $x,y\in\R^d$ at angle $\varphi\in[0,\pi]$, there exists $s$ such that $\|\pi_s(x)\|\geq \frac1{\sqrt d}\|x\|$, $\|\pi_s(y)\|\geq \frac1{\sqrt d}\|y\|$ and $\sin\varphi_s\geq \frac1{\sqrt d}\sin\varphi$, where $\varphi_s\in[0,\pi]$ is the angle between the vectors $\pi_s(x)$ and $\pi_s(y)$. 
\end{lemma}

\section{Proof of the upper bound}\label{sec:upper-bound}
In this section, we prove the upper bound in Theorem~\ref{thm:main}. 

\subsection{Union bound}
There is a covering of $\partial B(r)$ by at most $C\big(\tfrac{r}{\delta_r}\big)^{d-1}$ balls of radius $\delta_r$ centered on $\partial B(r)$ for some $C=C(d)$. 
If $\partial B(r)$ is visible from $0$ via $N$ points, then one of the balls in the covering is visible from $0$ via $N$ points in $B(r)$. 
Thus, by rotational invariance of $\mathcal C$, 
\[
P_{\mathrm{vis}}^N(r) 
\leq C\big(\tfrac{r}{\delta_r}\big)^{d-1}\,
\mathsf P\big[\text{$B(re_1,\delta_r)$ is visible from $0$ via $N$ points in $B(r)$}\big]
\]
Furthermore, 
\begin{multline*}
\mathsf P\big[\text{$B(re_1,\delta_r)$ is visible from $0$ via $N$ points in $B(r)$}\big]\leq \mathsf P\big[\text{$B(re_1,\delta_r)$ is visible from $0$}\big]\\
+ \sum_{\kappa=0}^\infty
\mathsf P\Big[
\begin{array}{c}
\text{$B(re_1,\delta_r)$ is visible from $0$ via $N$ points in $B(r)\cap B(\R e_1,(\kappa+1)h_r)$}\\
\text{but not visible via $N$ points in $B(r)\cap B(\R e_1,\kappa h_r)$}
\end{array}
\Big]. 
\end{multline*}
By \cite{MS-Visibility-AIHP}, 
$\mathsf P\big[\text{$B(re_1,\delta_r)$ is visible from $0$}\big]\leq Cf(r)$. To bound the probability in the sum, note that there is a covering of $B(r)\cap B(\R e_1,(\kappa+1) h_r)$ by at most 
\[
C\frac{r((\kappa+1) h_r)^{d-1}}{\beta_r \delta_r^{d-1}}
\]
cylinders $B([o, o+\beta_r e_1],\delta_r)$ with $o\in B(r)$, and if $B(re_1,\delta_r)$ is visible from $0$ via $N$ points in $B(r)\cap B(\R e_1,(\kappa+1) h_r)$ but not in $B(r)\cap B(\R e_1,\kappa h_r)$, then there exist $N$ cylinders $C_1,\ldots, C_N$ in the covering such that 
the event $\mathrm{vis}\big(0,C_1,\ldots, C_N, B(re_1,\delta_r)\big)$ occurs and at least one of the cyinders is not contained in $B(\R e_1,\kappa h_r)$. 
Thus, the probability in the sum is bounded from above by 
\[
C\Big(\frac{r((\kappa+1) h_r)^{d-1}}{\beta_r \delta_r^{d-1}}\Big)^N
\max_{C_1,\ldots, C_N} \mathsf P\big[\mathrm{vis}\big(0,C_1,\ldots, C_N, B(re_1,\delta_r)\big)\big],
\]
where the maximum is over cylinders $C_1,\ldots, C_N$ from the covering of $B(r)\cap B(\R e_1,(\kappa+1) h_r)$ at least one of which is not contained in $B(r)\cap B(\R e_1,\kappa h_r)$. Therefore, the result will follow if we prove 
\begin{lemma}\label{l:visibility-upper-main}
There exist positive constants $c=c(d,N,\mathrm{law}(\mathcal C))$ and $C=C(d,N,\mathrm{law}(\mathcal C))$ such that for any 
$\kappa\in \N_0$ and $o_1,\ldots, o_N\in B(r)$ not all contained in $B(\R e_1,\kappa h_r)$, if $C_i = B([o_i, o_i+\beta_r e_1],\delta_r)$, for $1\leq i\leq N$, then 
\begin{equation}\label{eq:visibility-upper-main}
\mathsf P\big[\mathrm{vis}\big(0,C_1,\ldots, C_N, B(re_1,\delta_r)\big)\big] \leq 
C e^{-c\kappa} f(r).
\end{equation}
\end{lemma}

Indeed, by \eqref{eq:visibility-upper-main}, 
\begin{eqnarray*}
P_{\mathrm{vis}}^N(r) 
&\leq &C\Big(\frac{r}{\delta_r}\Big)^{d-1}f(r)  + C\Big(\frac{r}{\delta_r}\Big)^{d-1}\Big(\frac{r h_r^{d-1}}{\beta_r \delta_r^{d-1}}\Big)^N f(r) \sum_{\kappa=1}^\infty (\kappa+1)^{N(d-1)} e^{-c\kappa}\\
&\leq &C\Big(\frac{r}{\delta_r}\Big)^{d-1}\Big(\frac{r h_r^{d-1}}{\beta_r \delta_r^{d-1}}\Big)^N f(r),
\end{eqnarray*}
and the result follows from the definition of the scales $\delta_r$, $\beta_r$ and $h_r$. 
It remains to prove Lemma~\ref{l:visibility-upper-main}. 

\subsection{Proof of Lemma~\ref{l:visibility-upper-main}}
In this section, we reduce Lemma~\ref{l:visibility-upper-main} to some general statements about visibility probabilities, which will then be proven for each of the three models separately. The arguments in this section apply to all three models. 

\smallskip

Fix cylinders $C_1,\ldots, C_N$ as in Lemma~\ref{l:visibility-upper-main} and assume, without loss of generality, that $o_1,\ldots, o_N$ are contained in $B(\R e_1, (\kappa+1)h_r)$. We also write $C_0 = \{0\}$ and $C_{N+1} = B(re_1,\delta_r)$, $o_0 = 0$ and $o_{N+1} = re_1$. We denote by $L_i$ the line segment $[o_{i-1},o_i]$. Let 
\[
I = \{1\leq i\leq N+1\,:\, |L_i|\geq 2\}.
\]
For each $i\in I$, let $p_i$ be the point on $L_i$ at distance $1$ from $o_{i-1}$ and $p_i'$ the point on $L_i$ at distance $1$ from $o_i$, and let $L_i'$ be the line segment $[p_i,p_i']$.
Let 
\[
q_i = d(o_i + \beta_re_1, L_i^\infty) + \delta_r,\,\,1\leq i\leq N+1,
\]
where $L_i^\infty$ is the line containing $L_i$. 
The following lemma collects key properties of $q_i$, in particular, its second statement gives the key relation between the scales $\delta_r$, $\beta_r$ and $h_r$. 
\begin{lemma}\label{l:qi-properties}
The following statements hold for all $\kappa\geq 0$ and $r\geq r_0(d)$: 
\begin{itemize}
\item[(a)]
For $i\in I$, every line segment connecting $C_{i-1}$ and $C_i$ intersects $B(p_i,q_i)$ and $B(p_i', q_i)$. In particular, if $C_i$ is visible from $C_{i-1}$ then $B(p_i',q_i)$ is visible from $B(p_i,q_i)$. 
\item[(b)] 
For $1\leq i\leq N+1$, 
\[
\frac{q_i}{\delta_{|L_i|}} \leq 10(\kappa+1). 
\]
\end{itemize}
\end{lemma}
\begin{proof}[Proof of Lemma~\ref{l:qi-properties}]
Note that $B(p_i,q_i)\cap C_{i-1} = \emptyset$, $B(p_i',q_i)\cap C_i=\emptyset$, and for all $x\in C_{i-1}\cup C_i$, 
\[
d(x,L_i^\infty) \leq d(o_i+\beta_re_1,L_i^\infty) + \delta_r = q_i.
\]
Thus, any line segment $[x,y]$ with $x\in C_{i-1}$ and $y\in C_i$ intersects both $B(p_i,q_i)$ and $B(p_i',q_i)$, and statement (a) follows. 

\smallskip

To prove (b), notice that by the similarity of triangles and the fact that $L_i$ is contained in $B(\R e_1, (\kappa+1)h_r)$, 
\[
\frac{d(o_i + \beta_re_1, L_i^\infty)}{\beta_r} \leq  \frac{2(\kappa+1)h_r}{|L_i|}.
\]
Since also $d(o_i + \beta_re_1, L_i^\infty)\leq\beta_r$, we obtain that
\[
d(o_i + \beta_re_1, L_i^\infty)\leq \min\Big(\beta_r, \frac{2(\kappa+1)h_r\beta_r}{|L_i|}\Big).
\]
If $|L_i|\geq 2(\kappa+1)h_r$, then in both cases \eqref{eq:delta-beta-h-1} and \eqref{eq:delta-beta-h-2}, 
\[
\frac{d(o_i + \beta_re_1, L_i^\infty)}{\delta_{|L_i|}}
\leq \frac{2(\kappa+1)h_r\beta_r}{|L_i|\delta_{|L_i|}}\leq 8(\kappa+1).
\]
If $|L_i|\leq 2(\kappa+1)h_r$, then in both cases \eqref{eq:delta-beta-h-1} and \eqref{eq:delta-beta-h-2}, 
\[
\frac{d(o_i + \beta_re_1, L_i^\infty)}{\delta_{|L_i|}}
\leq \frac{\beta_r}{\delta_{|L_i|}}\leq 8(\kappa+1),
\]
where in the case \eqref{eq:delta-beta-h-2} we have also used the monotonicity of the function $x\mapsto \frac{x}{(\log x)^2}$. It remains to notice that $\delta_r/\delta_{|L_i|}\leq \delta_r/\delta_{2r}\leq 2$, since $|L_i|\leq 2r$. 
\end{proof}

We proceed with the proof of Lemma~\ref{l:visibility-upper-main}.
By Lemma~\ref{l:qi-properties}(a), 
\[
\mathsf P\big[\mathrm{vis}\big(0,C_1,\ldots, C_N, B(re_1,\delta_r)\big)\big]
\leq 
\mathsf P\Big[\bigcap_{i\in I}\mathrm{vis}\big(B(p_i,q_i),B(p_i',q_i)\big)\Big].
\]
Lemma~\ref{l:visibility-upper-main} immediately follows from the following lemmas. 
\begin{lemma}\label{l:upper-bound-1}
There exists $C=C(d,N,\mathrm{law}(\mathcal C))$ such that for any $\kappa\in \N_0$ and any 
$o_1,\ldots, o_N\in B(r)\cap B(\R e_1,(\kappa+1) h_r)$,
\[
\mathsf P\Big[\bigcap_{i\in I}\mathrm{vis}\big(B(p_i,q_i),B(p_i',q_i)\big)\Big]
\leq 
e^{C(\kappa+1)}\,
\mathsf P\Big[\bigcap_{i\in I}\mathrm{vis}\big(p_i,p_i'\big)\Big].
\]
\end{lemma}
\begin{lemma}\label{l:upper-bound-2}
There exists $C=C(d,N,\mathrm{law}(\mathcal C))$ such that for any $o_1,\ldots, o_N\in B(r)$, 
\[
\mathsf P\Big[\bigcap_{i\in I}\mathrm{vis}\big(p_i,p_i'\big)\Big]\leq 
e^{C}\,\mathsf P\Big[\bigcap_{i=1}^{N+1}\mathrm{vis}\big(o_{i-1},o_i\big)\Big].
\]
\end{lemma}
\begin{lemma}\label{l:upper-bound-3}
There exist $c=c(d,N,\mathrm{law}(\mathcal C))$ and $C=C(d,N,\mathrm{law}(\mathcal C))$ such that for any 
$\kappa\in \N_0$ and any 
$o_1,\ldots, o_N\in B(r)$ not all contained in $B(\R e_1,\kappa h_r)$, 
\[
\mathsf P\Big[\bigcap_{i=1}^{N+1}\mathrm{vis}\big(o_{i-1},o_i\big)\Big]\leq e^{- c\kappa (\kappa\wedge  \log\log r) + C}f(r). 
\]
\end{lemma}
We will prove these lemmas separately for each of the three models in the next sections.

\subsection{Poisson-Boolean models}\label{sec:upper-bound-BM}
In this section, we prove Lemmas~\ref{l:upper-bound-1}--\ref{l:upper-bound-3} for the Poisson-Boolean models. 
Let $\omega = \sum_{k\geq 1}\delta_{(x_k,r_k)}$ be a Poisson point process on $\R^d\times \R_+$ ($d\geq 2$) with intensity measure $\alpha dx\otimes \mathsf Q$, and $\mathcal C =\mathcal C(\omega) = \bigcup_{k\geq 1}B(x_k,r_k)$. 

\begin{proof}[Proof of Lemma~\ref{l:upper-bound-1}]
Observe that for each $i\in I$, the event 
\[
\big\{\omega\,:\,\text{$B(p_i',q_i)$ is visible from $B(p_i,q_i)$ through the vacant set of $\mathcal C(\omega)$}\big\}
\]
implies the event 
\[
G_i = \big\{\omega = \sum_{k\geq 1}\delta_{(x_k,r_k)}\;:\;\text{for all $k$, $d(x_k,L_i')>r_k-q_i$}\big\}.
\]
Let 
\[
A = \big\{(x,r)\in\R^d\times\R_+\,:\,
\text{for some $i$, $d(x,L_i')\leq r-q_i$}
\big\},
\]
then by the definition of the Poisson point process, 
\begin{eqnarray*}
\mathsf P\big[\bigcap_{i\in I}G_i\big] 
&= &\mathsf P\big[\omega(A)=0\big]
= \exp\Big(-\alpha \int_A dx\mathsf Q(dr)\Big)\\
&= &\exp\Big(-\alpha \mathsf E\big[\lambda_d\big(\bigcup_{i\in I}B(L_i',(\rho-q_i)_+)\big)\big]\Big).
\end{eqnarray*}
On the other hand, 
\begin{eqnarray*}
\mathsf P^\alpha_\mathsf Q\Big[\bigcap_{i\in I}\mathrm{vis}\big(p_i,p_i'\big)\Big] 
&= &\mathsf P\Big[\omega = \sum_{k\geq 1}\delta_{(x_k,r_k)}\;:\;\text{for all $k$ and $i\in I$, $d(x_k,L_i')>r_k$}\Big]\\
&=
&\exp\Big(-\alpha \mathsf E\big[\lambda_d\big(\bigcup_{i\in I}B(L_i',\rho)\big)\big]\Big).
\end{eqnarray*}
Note that 
\begin{multline*}
\mathsf E\big[\lambda_d\big(\bigcup_{i\in I}B(L_i',\rho)\big)\big] - \mathsf E\big[\lambda_d\big(\bigcup_{i\in I}B(L_i',(\rho-q_i)_+)\big)\big]\\
\begin{aligned}
&\leq\,\, \mathsf E\big[\lambda_d\big(\bigcup_{i\in I}B(L_i',\rho\big)\setminus B(L_i',(\rho-q_i)_+)\big)\big]
\leq \sum_{i\in I} 
\mathsf E\big[\lambda_d\big(B(L_i',\rho\big)\setminus B(L_i',(\rho-q_i)_+)\big)\big]\\
&=\,\, \sum_{i\in I} 
\mathsf E\big[\lambda_d\big(B(L_i',\rho)\big)- \lambda_d\big(B(L_i',(\rho-q_i)_+)\big)\big]
= \sum_{i\in I}\kappa_{d-1}|L_i'|\,\mathsf E\big[\rho^{d-1} - (\rho-q_i)_+^{d-1}\big],
\end{aligned}
\end{multline*}
where in the last step we used Lemma~\ref{l:volume-cylinder}. 
Now, 
\begin{eqnarray*}
\mathsf E\big[\rho^{d-1} - (\rho-q_i)_+^{d-1}\big]
&\leq &\mathsf E\big[\rho^{d-1};\rho\leq q_i\big] + 
\mathsf E\big[\rho^{d-1} - (\rho-q_i)^{d-1}; \rho> q_i\big]\\
&\leq &q_i\mathsf E[\rho^{d-2}] + q_i(d-1)\mathsf E[\rho^{d-2}]
= q_i d\mathsf E[\rho^{d-2}].
\end{eqnarray*}
Finally, by Lemma~\ref{l:qi-properties} and the fact that $\delta_r = \frac1r$, 
\[
|L_i'|q_i \leq |L_i|q_i = \frac{q_i}{\delta_{|L_i|}} \leq 10(\kappa+1). 
\]
Thus, 
\[
\frac{\mathsf P^\alpha_\mathsf Q\Big[\bigcap_{i\in I}\mathrm{vis}\big(B(p_i,q_i),B(p_i',q_i)\big)\Big]}{
\mathsf P^\alpha_\mathsf Q\Big[\bigcap_{i\in I}\mathrm{vis}\big(p_i,p_i'\big)\Big]}
\leq 
\exp\Big(\alpha 10(\kappa+1)(N+1)\kappa_{d-1}d\mathsf E[\rho^{d-2}]\Big),
\]
as required.
\end{proof}

\begin{proof}[Proof of Lemma~\ref{l:upper-bound-2}]
By \eqref{eq:BM-mu},
\[
\mathsf P^\alpha_\mathsf Q\Big[\bigcap_{i\in I}\mathrm{vis}\big(p_i,p_i'\big)\Big]=
\mathsf P^\alpha_\mathsf Q\Big[\mathcal C\cap \big(\bigcup_{i\in I}L_i'\big)=\emptyset\Big]
=
\exp\Big(-\alpha \mathsf E\big[\lambda_d\big(\bigcup_{i\in I}B(L_i',\rho)\big)\big]\Big)
\]
and, similarly, 
\[
\mathsf P^\alpha_\mathsf Q\Big[\bigcap_{i=1}^{N+1}\mathrm{vis}\big(o_{i-1},o_i\big)\Big] 
=
\exp\Big(-\alpha \mathsf E\big[\lambda_d\big(\bigcup_{i=1}^{N+1}B(L_i,\rho)\big)\big]\Big).
\]
As in the proof of Lemma~\ref{l:upper-bound-1}, we obtain
\begin{multline*}
\mathsf E\big[\lambda_d\big(\bigcup_{i=1}^{N+1}B(L_i,\rho)\big)\big] - \mathsf E\big[\lambda_d\big(\bigcup_{i\in I}B(L_i',\rho)\big)\big]\\
\begin{aligned}
&\leq\,\, \sum_{i\notin I} \mathsf E\big[\lambda_d\big(B(L_i,\rho)\big)\big] + \sum_{i\in I}\Big(\mathsf E\big[\lambda_d\big(B(L_i,\rho)\big)\big] - \mathsf E\big[\lambda_d\big(B(L_i',\rho)\big)\big]
\Big)\\
&=\,\, \sum_{i\notin I}\big(\kappa_d\mathsf E[\rho^d] + \kappa_{d-1}\mathsf E[\rho^{d-1}]|L_i|\big) + 
\sum_{i\in I}\kappa_{d-1}\mathsf E[\rho^{d-1}]\big(|L_i| - |L_i'|\big)\\
&\leq\,\, (N+1)\big(\kappa_d\mathsf E[\rho^d] + 2 \kappa_{d-1}\mathsf E[\rho^{d-1}]\big), 
\end{aligned}
\end{multline*}
where the equality follows from Lemma~\ref{l:volume-cylinder}, and the last inequality from the fact that $\{|L_i|\}_{i\notin I}$ and $\{|L_i| - |L_i'|\}_{i\in I}$ are all bounded from above by $2$. 
Thus, 
\[
\frac{\mathsf P^\alpha_\mathsf Q\Big[\bigcap_{i\in I}\mathrm{vis}\big(p_i,p_i'\big)\Big]}{\mathsf P^\alpha_\mathsf Q\Big[\bigcap_{i=1}^{N+1}\mathrm{vis}\big(o_{i-1},o_i\big)\Big]}
\leq \exp\Big(\alpha(N+1)\big(\kappa_d\mathsf E[\rho^d] + 2 \kappa_{d-1}\mathsf E[\rho^{d-1}]\big)\Big),
\]
as required. 
\end{proof}

\begin{proof}[Proof of Lemma~\ref{l:upper-bound-3}]
By \eqref{eq:BM-mu}, 
\[
\mathsf P^\alpha_\mathsf Q\Big[\bigcap_{i=1}^{N+1}\mathrm{vis}\big(o_{i-1},o_i\big)\Big] 
=
\exp\Big(-\alpha \mathsf E\big[\lambda_d\big(\bigcup_{i=1}^{N+1}B(L_i,\rho)\big)\big]\Big)
\]
and 
\[
f(r) = \mathsf P^\alpha_\mathsf Q\big[\mathrm{vis}\big(0,re_1\big)\big] 
=
\exp\Big(-\alpha \mathsf E\big[\lambda_d\big(B([0,re_1],\rho)\big)\big]\Big).
\]
Thus, it suffices to prove that 
\[
\mathsf E\Big[\lambda_d\Big(\bigcup_{i=1}^{N+1}B(L_i,\rho)\Big)\Big]
- \mathsf E\big[\lambda_d\big(B([0,re_1],\rho)\big)\big]
\geq c\kappa^2 - C.
\]
By the assumption of the lemma, there exists $i_*$ such that 
the projection of $L_{i_*}$ on the hyperplane orthogonal to $e_1$ has length $\geq \frac1N\kappa h_r$. 
Let $\widehat L_{i_*}$ be the projection of $L_{i_*}$ onto the $Ox_1$-axis. 
Since $|L_{i_*}|\leq 2r$ and $h_r = \sqrt r$, 
\[
|L_{i_*}| - |\widehat L_{i_*}|  = \frac{|L_{i_*}|^2 - |\widehat L_{i_*}|^2}{|L_{i_*}| + |\widehat L_{i_*}|} \geq \frac{1}{4N^2}\kappa^2.
\]
Hence by Lemma~\ref{l:volume-cylinder}, 
\[
\lambda_d\big(B(L_{i_*},\rho)\big) - \lambda_d\big(B(\widehat L_{i_*},\rho)\big)
= \kappa_{d-1}\rho^{d-1}\big(|L_{i_*}| - |\widehat L_{i_*}| \big)
\geq \kappa_{d-1}\rho^{d-1}\frac{1}{4N^2}\kappa^2.
\]
Let $H_{a,b} = \{x\in\R^d\,:\,a\leq x(1)\leq b\}$. Note that for the minimal $H_{a,b}$ containing $B(L_{i_*},\rho)$, 
\[
\lambda_d\Big(H_{a,b}\cap\bigcup_{i=1}^{N+1}B(L_i,\rho)\Big)
\geq \lambda_d\big(B(L_{i_*},\rho)\big)
\]
and 
\[
\lambda_d\big(H_{a,b}\cap B([0,re_1],\rho)\big) 
\leq \lambda_d\big(B(\widehat L_{i_*},\rho)\big) + 2\rho\lambda_{d-1}\big(B(0,\rho)\big) 
= \lambda_d\big(B(\widehat L_{i_*},\rho)\big) + 2\kappa_{d-1}\rho^d,
\]
and for arbitrary $a\leq b$, 
\[
\lambda_d\Big(H_{a,b}\cap\bigcup_{i=1}^{N+1}B(L_i,\rho)\Big)
\geq \lambda_d\big(H_{a,b}\cap B([0,re_1],\rho)\big).
\]
Hence
\begin{equation}\label{eq:BM-upper-bound-3}
\mathsf E\Big[\lambda_d\Big(\bigcup_{i=1}^{N+1}B(L_i,\rho)\Big)\Big]
- \mathsf E\big[\lambda_d\big(B([0,re_1],\rho)\big)\big]
\geq
\kappa_{d-1}\mathsf E[\rho^{d-1}] \frac{1}{4N^2}\kappa^2 - 2\kappa_{d-1}\mathsf E[\rho^d],
\end{equation}
as required. 
\end{proof}

\subsection{Poisson cylinders}\label{sec:upper-bound-PC}
In this section, we prove Lemmas~\ref{l:upper-bound-1}--\ref{l:upper-bound-3} for the Poisson cylinders. 
It suffices to prove the statements for $r\geq r_0=r_0(d,N,\rho)$. 
Thus, we may assume throughout the proofs that $q_i<\rho$ for all $i$.

\begin{proof}[Proof of Lemma~\ref{l:upper-bound-1}]
Let $\omega = \sum_{i\geq 1}\delta_{L_i}$ be the Poisson lines process at level $\alpha$. For $\rho>0$, let $\mathcal C_\rho = \bigcup_{i\geq 1}B(L_i,\rho)$ and observe that for each $i\in I$, if $B(p_i',q_i)$ is visible from $B(p_i,q_i)$ through the vacant set of $\mathcal C_\rho$, then $p_i'$ is visible from $p_i$ through the vacant set of $\mathcal C_{\rho-q_i}$. Therefore, by \eqref{eq:PC-mu-rho}, it suffices to prove that 
\[
\int\limits_{SO_d}\lambda_{d-1}\Big(\pi\big(B\big(\bigcup_{i\in I}L_i',\rho\big)_\phi\big)
\Big)\nu(d\phi)
- 
\int\limits_{SO_d}\lambda_{d-1}\Big(\pi\big(B\big(\bigcup_{i\in I}L_i',\rho-q_i\big)_\phi\big)
\Big)\nu(d\phi)
\leq C(\kappa+1).
\]
By Lemma~\ref{l:volume-cylinder}, 
\begin{multline*}
\lambda_{d-1}\Big(\pi\big(B\big(\bigcup_{i\in I}L_i',\rho\big)_\phi\big)
\Big) - \lambda_{d-1}\Big(\pi\big(B\big(\bigcup_{i\in I}L_i',\rho-q_i\big)_\phi\big)
\Big)\\
\begin{aligned}
&=\,\, 
\lambda_{d-1}\Big(\bigcup_{i\in I}B(\pi((L_i')_\phi),\rho)\cap H\Big)
- 
\lambda_{d-1}\Big(\bigcup_{i\in I}B(\pi((L_i')_\phi),\rho-q_i)\cap H\Big)\\
&\leq\,\,
\sum_{i\in I}\Big(
\lambda_{d-1}\big(B(\pi((L_i')_\phi),\rho)\cap H\big)
- 
\lambda_{d-1}\big(B(\pi((L_i')_\phi),\rho-q_i)\cap H\big)
\Big)\\
&=\,\,
\sum_{i\in I}\Big(
\kappa_{d-1}\mathsf E[\rho^{d-1} - (\rho-q_i)^{d-1}] + \kappa_{d-2}|\pi((L_i')_\phi)|\mathsf E[\rho^{d-2} - (\rho-q_i)^{d-2}]
\Big)\\
&\leq\,\,
(N+1)\kappa_{d-1}\mathsf E[\rho^{d-1}] + 
\kappa_{d-2}(d-1)\mathsf E[\rho^{d-3}]
\sum_{i\in I}|L_i|q_i\\
&\leq\,\,
(N+1)\kappa_{d-1}\mathsf E[\rho^{d-1}] + 
10(\kappa+1)(N+1)\kappa_{d-2}(d-1)\mathsf E[\rho^{d-3}],
\end{aligned}
\end{multline*}
as required, where the last inequality follows from Lemma~\ref{l:qi-properties} and the fact that $\delta_r = \frac1r$. 
\end{proof}

\begin{proof}[Proof of Lemma~\ref{l:upper-bound-2}]
By \eqref{eq:PC-mu-rho}, it suffices to prove that 
\[
\int\limits_{SO_d}\lambda_{d-1}\Big(\pi\big(B\big(\bigcup_{i=1}^{N+1}L_i,\rho\big)_\phi\big)\Big)\nu(d\phi)
- 
\int\limits_{SO_d}\lambda_{d-1}\Big(\pi\big(B\big(\bigcup_{i\in I}L_i',\rho\big)_\phi\big)\Big)\nu(d\phi)
\leq C.
\]
By a similar computation as in the proof of the lemma for the Poisson-Boolean models, we obtain that for each $\phi\in SO_d$, 
\begin{multline*}
\lambda_{d-1}\Big(\pi\big(B\big(\bigcup_{i=1}^{N+1}L_i,\rho\big)_\phi\big)\Big)
- 
\lambda_{d-1}\Big(\pi\big(B\big(\bigcup_{i\in I}L_i',\rho\big)_\phi\big)\Big)\\
\begin{aligned}
&=\,\,
\lambda_{d-1}\Big(\bigcup_{i=1}^{N+1}B(\pi((L_i)_\phi),\rho)\cap H\Big)
- 
\lambda_{d-1}\Big(\bigcup_{i\in I}B(\pi((L_i')_\phi),\rho)\cap H\Big)\\
&\leq\,\,
(N+1)\big(\kappa_{d-1}\mathsf E[\rho^{d-1}] + 2 \kappa_{d-2}\mathsf E[\rho^{d-2}]\big), 
\end{aligned}
\end{multline*}
as required. 
\end{proof}

\begin{proof}[Proof of Lemma~\ref{l:upper-bound-3}]
By \eqref{eq:PC-mu-rho} and the definition of $f(r)$, it suffices to prove that 
\[
\int_{SO_d}\lambda_{d-1}\Big(\pi\big(B\big(\bigcup_{i=1}^{N+1}L_i,\rho\big)_\phi\big)\Big)
\nu(d\phi)
- \int_{SO_d}\lambda_{d-1}\Big(\pi\big(B\big([0,re_1],\rho\big)_\phi\big)\Big)
\nu(d\phi)\geq c\kappa^2 - C.
\]
Let $\pi_s$ be the projection on the hyperplane $H_s=\{z\,:\,z(s)=0\}$. By the rotational invariance of the Haar measure $\nu$, the difference of the integrals above equals to 
\begin{multline*}
\frac1d\sum_{s=1}^d
\int_{SO_d}\Big[\lambda_{d-1}\Big(\pi_s\big(B\big(\bigcup_{i=1}^{N+1}L_i,\rho\big)_\phi\big)\Big)
- \lambda_{d-1}\Big(\pi_s\big(B\big([0,re_1],\rho\big)_\phi\big)\Big)
\Big]\nu(d\phi)\\
= 
\int_{SO_d}\frac1d\sum_{s=1}^d\Big[\lambda_{d-1}\Big(\bigcup_{i=1}^{N+1}B(\pi_s((L_i)_\phi),\rho)\cap H_s\Big) - \lambda_{d-1}\Big(B([0,\pi_s(\phi(re_1))],\rho)\cap H_s\Big)\Big]\nu(d\phi).
\end{multline*}
Note that $\pi_s((L_1)_\phi),\ldots, \pi_s((L_{N+1})_\phi)$ are the line segments in the hyperplane $H_s$ with endpoints $\pi_s(\phi(o_0)) = 0$, $\pi_s(\phi(o_1)),\ldots, \pi_s(\phi(o_N))$, $\pi_s(\phi(o_{N+1})) = \pi_s(\phi(re_1))$. Thus, the difference of the volumes under the integral is non-negative for each $s$. 

Furthermore, by the assumption of the lemma, there exists $i_0$ such that the distance of $o_{i_0}$ from the line $\R e_1$ is at least $\kappa h_r$.
Let $\varphi$ be the angle between the vectors $o_{i_0}$ and $re_1$ and let $\varphi_s$ be the angle between their rotated projections $\pi_s(\phi(o_{i_0}))$ and $\pi_s(\phi(re_1))$. 
By Lemma~\ref{l:vector-projections}, for any $\phi\in SO_d$, there exists $1\leq s_0\leq d$ such that 
\[
\|\pi_{s_0}(\phi(o_{i_0}))\|\geq \frac1{\sqrt d}\|o_{i_0}\|,\,\,
\|\pi_{s_0}(\phi(re_1)\|\geq \frac1{\sqrt d}\|re_1\|,\,\,
\sin\varphi_{s_0} \geq \frac1{\sqrt d}\sin\varphi.
\]
Hence, the distance of $\pi_{s_0}(\phi(o_{i_0}))$ from the line $\R\pi_{s_0}(\phi(e_1))$ is 
\[
d\big(\pi_{s_0}(\phi(o_{i_0})),\R\pi_{s_0}(\phi(e_1))\big) = \|\pi_{s_0}(\phi(o_{i_0}))\| \sin\varphi_{s_0} \geq \frac1d \|o_{i_0}\|\sin\varphi
= \frac1d d(o_{i_0},\R e_1)\geq \frac1d \kappa h_r. 
\]
Now, by the estimates of the volume difference from the proof for the Poisson-Boolean model (see \eqref{eq:BM-upper-bound-3}), we obtain that 
\begin{multline*}
\lambda_{d-1}\Big(\pi_{s_0}\Big(\bigcup_{i=1}^{N+1}B(\phi(L_i),\rho)\Big)\Big) - \lambda_{d-1}\Big(\pi_{s_0}\big(B([0,\phi(re_1)],\rho)\big)\Big)\\
\geq \kappa_{d-2}\mathsf E[\rho^{d-2}] \frac{1}{4d^2N^2}\kappa^2 - 2\kappa_{d-2}\mathsf E[\rho^{d-1}] =: c\kappa^2 - C.
\end{multline*}
Thus, for every rotation $\phi\in SO_d$, 
\[
\frac1d\sum_{s=1}^d\Big[\lambda_{d-1}\Big(\bigcup_{i=1}^{N+1}B(\pi_s((L_i)_\phi),\rho)\cap H_s\Big) - \lambda_{d-1}\Big(B([0,\pi_s(\phi(re_1))],\rho)\cap H_s\Big)\Big]
\geq \frac1d(c \kappa^2 - C),
\]
and the result follows.
\end{proof}

\subsection{Proof of Lemmas~\ref{l:upper-bound-1} and \ref{l:upper-bound-2} for Brownian interlacements}\label{sec:upper-bound-BI-1}
In this section, we prove Lemmas~\ref{l:upper-bound-1} and \ref{l:upper-bound-2} for the Brownian interlacements. 
It suffices to prove the statements for $r\geq r_0=r_0(d,N,\rho)$. 
Thus, we may assume throughout the proofs that $q_i<\frac12\rho$ for all $i$.  

\begin{proof}[Proof of Lemma~\ref{l:upper-bound-1}]
Let $\omega = \sum_{i\geq 1}\delta_{w_i^*}$ be the Brownian interlacement point process at level $\alpha$. For $\rho>0$, let $\mathcal C_\rho = \bigcup_{i\geq 1}B(\mathrm{range}(w_i^*),\rho)$ and observe that for each $i\in I$, if $B(p_i',q_i)$ is visible from $B(p_i,q_i)$ through the vacant set of $\mathcal C_\rho$, then $p_i'$ is visible from $p_i$ through the vacant set of $\mathcal C_{\rho-q_i}$. 
Therefore, by \eqref{eq:BI-capacity-rho}, it suffices to prove that 
\[
\mathrm{cap}\Big(\bigcup_{i\in I}B(L_i',\rho)\Big)
- 
\mathrm{cap}\Big(\bigcup_{i\in I}B(L_i',\rho-q_i)\Big)
\leq C(\kappa+1).
\]
Let $A_i = B(L_i',\rho-q_i)$ and $B_i = B(L_i',\rho)$. By the definition of the capacity, for all $R$ large enough, 
\begin{multline*}
\mathrm{cap}\Big(\bigcup_{i\in I}B_i\Big)
- 
\mathrm{cap}\Big(\bigcup_{i\in I}A_i\Big)
= \mathsf P_{\mu_R}\big[H_{\bigcup_{i\in I}B_i}<\infty\big] - \mathsf P_{\mu_R}\big[H_{\bigcup_{i\in I}A_i}<\infty\big]\\
\begin{aligned}
&=\,\, \mathsf P_{\mu_R}\big[H_{\bigcup_{i\in I}B_i}<\infty, H_{\bigcup_{i\in I}A_i}=\infty\big]
\leq \sum_{i\in I} \mathsf P_{\mu_R}\big[H_{B_i}<\infty, H_{A_i}=\infty\big]\\
&=\,\, \sum_{i\in I}\big( \mathsf P_{\mu_R}\big[H_{B_i}<\infty\big] - \mathsf P_{\mu_R}\big[H_{A_i}<\infty\big]\big)
= \sum_{i\in I}\big(\mathrm{cap}(B_i) - \mathrm{cap}(A_i)\big).
\end{aligned}
\end{multline*}
Thus, it suffices to prove that for all $i\in I$, 
\[
\mathrm{cap}\big(B(L_i',\rho)\big)
- 
\mathrm{cap}\big(B(L_i',\rho-q_i)\big)
\leq C(\kappa+1).
\]
Let $M_i= \lfloor q_i/\delta_{|L_i'|}\rfloor$. By Lemma~\ref{l:qi-properties}, $M_i\leq 10(\kappa+1)$. Furthermore, by \cite[(3.9)]{MS-Visibility-AIHP}, 
\[
\mathrm{cap}\big(B([0,re_1],\rho')\big)
- 
\mathrm{cap}\big(B([0,re_1],\rho'-\delta_r)\big)
\leq C',
\]
where $C'=C'(d,\rho')$ is a continuous function of $\rho'$. Thus, $C= \sup_{\frac12\rho\leq \rho'\leq \rho} C'(d,\rho')<\infty$ and 
\begin{multline*}
\mathrm{cap}\big(B(L_i',\rho)\big) - \mathrm{cap}\big(B(L_i',\rho-q_i)\big)\\
\begin{aligned}
&\leq\,\, \sum_{k=0}^{M_i} 
\big(\mathrm{cap}\big(B(L_i',\rho-k\delta_{|L_i'|})\big) - \mathrm{cap}\big(B(L_i',\rho-(k+1)\delta_{|L_i'|})\big)\big)\\
&\leq\,\, C(M_i+1)\leq 11C(\kappa+1),
\end{aligned}
\end{multline*}
as required. 
\end{proof}

\begin{proof}[Proof of Lemma~\ref{l:upper-bound-2}]
 By \eqref{eq:BI-capacity-rho}, it suffices to prove that 
\[
\mathrm{cap}\Big(\bigcup_{i=1}^{N+1}B(L_i,\rho)\Big)
- 
\mathrm{cap}\Big(\bigcup_{i\in I}B(L_i',\rho)\Big)
\leq C.
\]
By the subadditivity of the capacity, the above difference is at most 
\[
\sum_{i=1}^{N+1}\mathrm{cap}\big(B(L_i\setminus L_i',\rho)\big),
\]
where we set $L_i'=\emptyset$ for $i\notin I$. Since $B(L_i\setminus L_i',\rho)\subseteq B(o_{i-1},\rho+1)\cup B(o_i,\rho+1)$, $\mathrm{cap}\big(B(L_i\setminus L_i',\rho)\big)\leq C$ for some $C=C(d,\rho)$. The result follows. 
\end{proof}

\subsection{Proof of Lemma~\ref{l:upper-bound-3} for Brownian interlacements}\label{sec:upper-bound-BI-2}

By \eqref{eq:BI-capacity-rho} and the definition of $f(r)$, we need to show that 
\[
\mathrm{cap}\Big(\bigcup_{i=1}^{N+1}B(L_i,\rho)\Big)
- \mathrm{cap}\big(B([0,re_1],\rho)\big)
\geq c\kappa \big(\kappa\wedge \log\log r\big) - C.
\]
Denote by $\theta_i$ the angle in $[0,\frac\pi2]$ between the line segment $L_i$ and $\R e_1$ and consider the index sets
\begin{equation}\label{eq:I0-I1-I2}
\begin{aligned}
I_0 &= \Big\{1\leq i\leq N+1\,:\,|L_i|\sin\theta_i \geq \frac{\kappa h_r}{5(N+1)^2}\Big\},\\
I_1 &= \Big\{i\in I_0\,:\,\theta_i\leq \frac\pi4\Big\},\\
I_2 &= \Big\{i\in I_0\,:\,\theta_i>\frac\pi4, |L_i|\cos\theta_i\geq \kappa \log^2 r\Big\}.
\end{aligned}
\end{equation}
We distinguish two cases: $I_1\cup I_2\neq\emptyset$ and $I_1\cup I_2=\emptyset$. 
\begin{proof}[Proof of Lemma~\ref{l:upper-bound-3} in the case $I_1\cup I_2\neq\emptyset$]
Fix $i_*\in I_1\cup I_2$. 
From the sequence of line segments 
\[
(L_{(1)},\ldots L_{(N+1)}) = \big(L_{i_*}, L_{i_*+1},\ldots, L_{N+1}, L_{i_*-1},L_{i_*-2},\ldots, L_1\big), 
\]
we select a family of their subsegments $L_1', \ldots, L_{N'}'$ with pairwise disjoint projections on the $Ox_1$-axis as follows. 
Let $\mathcal L_1' = L_{(1)}$.
For $1\leq k\leq N$, let $B_k$ be the set of points in $\R^d$ whose projection on $Ox_1$-axis is at distance smaller than $4\rho$ from the projection of $\mathcal L_k'$ and define $\mathcal L_{k+1}' = L_{(k+1)}\setminus (B_1\cup\ldots\cup B_k)$. 
Note that $\{\mathcal L_1',\ldots, \mathcal L_{N+1}'\}$ is the set of $N'\leq N+1$ line segments. We order and denote these line segments by $L_1', \ldots, L_{N'}'$ so that the projection of $L_i'$ on the $Ox_1$-axis is to the left of the projection of $L_{i+1}'$ for each $i$, and let $i_*'$ be such that $L_{i_*'}' = L_{i_*}$. This way, the distance between the projections of $L_{i-1}'$ and $L_i'$ is precisely $4\rho$.

\smallskip

Let $r'$ be the sum of the lengths of the projections of $L_1',\ldots, L_{N'}'$ on the $Ox_1$-axis. Note that $r' \geq r - 4\rho (N+1)$, in particular (e.g.\ by \eqref{eq:capacity-asymptotics}), for some $C=C(d,\rho,N)$,
\[
\mathrm{cap}\big(B([0,re_1],\rho)\big) - \mathrm{cap}\big(B([0,r'e_1],\rho)\big)\leq C.
\]
Hence, by the monotonicity of the capacity, it suffices to prove that 
\[
\mathrm{cap}\Big(\bigcup_{i=1}^{N'}B(L_i',\rho)\Big)
- \mathrm{cap}\big(B([0,r'e_1],\rho)\big)
\geq c\kappa(\kappa\wedge \log\log r') - C.
\]
\emph{By abuse of notation, we omit the primes and write in the following $L_i$ instead of $L_i'$, $N$ instead of $N'$, $r$ instead of $r'$, and $i_*$ instead of $i_*'$.}

In the new notation, let $\mathcal L_0 = [0,re_1]$ and $\mathcal L_1 = \bigcup_{i=1}^NL_i$. 
Consider the points $a_1=0$, 
\[
a_i = \big(\sum_{k=1}^{i-1} r_k, 0,\ldots,0\big),\text{ for }2\leq i\leq N+1,
\]
and define the sets 
\[
A_i = \big\{x\in B(\mathcal L_0,\rho)\,:\, a_i(1)\leq x(1)<a_{i+1}(1)\big\},\,\text{ for }1\leq i\leq N.
\]
Note that $A_i$ are disjoint subsets of $B(\mathcal L_0,\rho)$ and by the subadditivity of the capacity, 
\[
\mathrm{cap}\big(B(\mathcal L_0,\rho)\big) - \mathrm{cap}\big(\bigcup_{i=1}^N A_i\big)\leq C.
\]
Hence it suffices to prove that 
\[
\mathrm{cap}\big(B(\mathcal L_1,\rho)\big) - \mathrm{cap}\big(\bigcup_{i=1}^N A_i\big)\geq c\kappa(\kappa\wedge \log\log r).
\]

Let $\mu_0$ be the normalized equilibrium measure of $\bigcup_{i=1}^N A_i$. By \eqref{eq:capacity-variational}, it suffices to prove that for some probability measure $\mu_1$ on $B(\mathcal L_1,\rho)$, 
\[
\frac{1}{I(\mu_1)} - \frac{1}{I(\mu_0)} \geq c\kappa(\kappa\wedge \log\log r), 
\]
where $I(\mu) = \iint \|x-y\|^{2-d}\mu(dx)\mu(dy)$, for which it suffices to show that 
\begin{equation}\label{eq:BI-capacities-I}
I(\mu_0) - I(\mu_1) \geq c\kappa(\kappa\wedge \log\log r) I(\mu_0)^2. 
\end{equation}
We identify $\mu_1$ as a push-forward of $\mu_0$ by some function $g: \bigcup_{i=1}^N A_i\to B(\mathcal L_1,\rho)$.
Let $\theta_i$ be the angle in $[0,\frac\pi2]$ between $L_i$ and $\R e_1$. Let $r_i = |L_i|\cos\theta_i$ be the length of the projection of $L_i$ onto the $Ox_1$-axis; by definition, $\sum_{i=1}^Nr_i = r$. 
Let $b_i$ and $b_i'$ be the end-vertices of the line segment $L_i$ so that $b_i(1)\leq b_i'(1)$. Fix some rotations $\psi_1, \ldots, \psi_N$ of $\R^d$ such that 
\[
\overrightarrow{b_ib_i'} = |L_i|\psi_i(e_1)
\]
and define the function $g:\bigcup_{i=1}^N A_i\to B(\mathcal L_1,\rho)$ by
\[
g(x) = b_i + \psi_i\Big(\big(\frac{|L_i|}{r_i}\big(x(1)-a_i(1)\big),x(2),\ldots, x(d)\big)\Big),\,\, \text{for } x\in A_i.
\]
In words, the function $g$ applied to $A_i$ is the composition of the following maps: (a) translation which maps $a_i$ to the origin, (b) stretching in the direction of $e_1$ by the factor $\frac{|L_i|}{r_i}$, (c) rotation by $\psi_i$, and (d) translation which maps $0$ to $b_i$; in particular, $g(a_i) = b_i$, $g(a_{i+1}-) = b_i'$ and $g(A_i)\subseteq B(L_i,\rho)$, for all $1\leq i\leq N$, where $g(a_{i+1}-)$ is the limit of $g(x)$ as $x\to a_{i+1}$ for $x\in A_i$. 

\smallskip

Now, consider the measure $\mu_1(\cdot) = \mu_0(g^{-1}(\cdot))$. Note that $\mu_1$ is a probability measure on $B(\mathcal L_1,\rho)$. Furthermore, by the change of variables formula, 
\[
I(\mu_1) = \iint \|x_1-y_1\|^{2-d}\mu_1(dx_1)\mu_1(dy_1) = \iint \|g(x)-g(y)\|^{2-d}\mu_0(dx)\mu_0(dy).
\]
Thus, 
\[
I(\mu_0) - I(\mu_1) = \iint \frac{\|g(x)-g(y)\|^{d-2}-\|x-y\|^{d-2}}{\|x-y\|^{d-2}\|g(x)-g(y)\|^{d-2}}\mu_0(dx)\mu_0(dy).
\]
We bound the integral with the help of the following lemma. 
\begin{lemma}\label{l:gx-gy-case1}
If $x,y\in A_i$ for some $1\leq i\leq N$, then 
\[
\|g(x)-g(y)\|^2 - \|x-y\|^2 = \tan^2\theta_i |x(1)-y(1)|^2.
\]
If $x\in A_i$ and $y\in A_j$ for some $1\leq i<j\leq N$, then 
\[
\|g(x)-g(y)\|\geq |\langle g(x)-g(y), e_1\rangle| \geq \|x-y\|.
\]
\end{lemma}
\begin{proof}
If $x,y\in A_i$ for some $1\leq i\leq N$, then 
\[
\|g(x)-g(y)\|^2 - \|x-y\|^2 = \frac{|L_i|^2}{r_i^2}|x(1)-y(1)|^2 - |x(1)-y(1)|^2 = \tan^2\theta_i |x(1)-y(1)|^2.
\]
To prove the second statement, note that for all $1\leq i<j\leq N$,
\[
\langle g(a_j),e_1\rangle-\langle g(a_i),e_1\rangle
= \|a_j-a_i\|+4\rho(j-i),
\]
and for $x=(x(1),0,\ldots, 0)\in A_i$ ($1\leq i\leq N$),
\[
\langle g(x) - g(a_i),e_1\rangle = \|x-a_i\|.
\]
Hence if $x=(x(1),0,\ldots, 0)\in A_i$ and $y=(y(1),0,\ldots,0)\in A_j$ with $i<j$, then 
\begin{eqnarray*}
\langle g(y)-g(x), e_1\rangle
&= &\langle g(a_j), e_1\rangle -\langle g(a_i),e_1\rangle 
+ \langle g(y)-g(a_j), e_1\rangle -\langle g(x) - g(a_i),e_1\rangle\\
&\geq &\|y-x\|+4\rho.
\end{eqnarray*}
Finally, for general $x\in A_i$ and $y\in A_j$ with $i<j$, writing $x'=\big(x(1),0,\ldots,0\big)$ and $y' = \big(y(1),0,\ldots,0\big)$ and using the previous case, we obtain that 
\[
|\langle g(x)-g(y), e_1\rangle|
\geq |\langle g(x')-g(y'), e_1\rangle| - 2\rho
\geq \|x'-y'\|+2\rho \geq \|x-y\|.\qedhere
\]
\end{proof}

\smallskip

It follows from Lemma~\ref{l:gx-gy-case1} that $I(\mu_1)\leq I(\mu_0)$, so that \eqref{eq:BI-capacities-I} holds for $\kappa=0$, and we can assume that $\kappa\geq 1$. Furthermore, 
\[
I(\mu_0) - I(\mu_1) \geq \iint\limits_{A_{i_*}\times A_{i_*}} \frac{\|g(x)-g(y)\|^{d-2}-\|x-y\|^{d-2}}{\|x-y\|^{d-2}\|g(x)-g(y)\|^{d-2}}\mu_0(dx)\mu_0(dy).
\]
Let $\widetilde A = \{(x,y)\in A_{i_*}\times A_{i_*}\,:\,|x(1)-y(1)|\geq 1\}$. Then for any $(x,y)\in \widetilde A$, 
\[
\|x-y\|^2 \leq |x(1)-y(1)|^2+ 4\rho^2\leq (1+4\rho^2)|x(1)-y(1)|^2,
\]
hence
\[
\|g(x)-g(y)\|^2 \geq \big(1 + \frac{\tan^2\theta_{i_*}}{1+4\rho^2}\big)\|x-y\|^2\geq \big(1 + \frac{\sin^2\theta_{i_*}}{1+4\rho^2}\big)\|x-y\|^2.
\]
Thus, 
\begin{eqnarray*}
I(\mu_0) - I(\mu_1) &\geq &\iint\limits_{\widetilde A} \frac{\|g(x)-g(y)\|^{d-2}-\|x-y\|^{d-2}}{\|x-y\|^{d-2}\|g(x)-g(y)\|^{d-2}}\mu_0(dx)\mu_0(dy)\\
&\geq &\Big(1 - \big(1 + \frac{\sin^2\theta_{i_*}}{1+4\rho^2}\big)^{\frac{2-d}2}\Big)
\iint\limits_{\widetilde A} \|x-y\|^{2-d}\mu_0(dx)\mu_0(dy)\\
&\geq &\big(1 - 2^{\frac{2-d}{2}}\big)\frac{\sin^2\theta_{i_*}}{1+4\rho^2}\iint\limits_{\widetilde A} \|x-y\|^{2-d}\mu_0(dx)\mu_0(dy),
\end{eqnarray*}
where the last inequality follows from the concavity of the function $f(u) = 1 - (1+u)^{(2-d)/2}$ for $u\in[0,1]$. 

Now, for any $x\in A_{i_*}$, if $S_k = \{y\in A_{i_*}\,:\, |x(1)-y(1)|\in[k,k+1)\}$, then by Lemma~\ref{l:equilibrium-measure-cylinder}, $\mu_0(S_k) \geq c/r$ for some $c=c(d,\rho)$, therefore,  
\[
\int\limits_{y\in A_{i_*}:\, (x,y)\in\widetilde A} \|x-y\|^{2-d}\mu_0(dy)
\geq \sum_{k=1}^{\lfloor r_{i_*}/3\rfloor}\int_{S_k}\|x-y\|^{2-d}\mu_0(dy)
\geq \frac{c}{r}\big(\log r_{i_*}\mathds{1}_{d=3} + \mathds{1}_{d\geq 4}\big),
\]
for some $c=c(d,\rho)$. Hence 
\begin{eqnarray}\label{eq:Imu0-Imu1-case1}
I(\mu_0) - I(\mu_1) 
&\geq &c\frac{\sin^2\theta_{i_*}}{r}\big(\log r_{i_*}\mathds{1}_{d=3} + \mathds{1}_{d\geq 4}\big)\mu_0(A_{i_*})\\\nonumber
&\geq &c\frac{r_{i_*}\sin^2\theta_{i_*}}{r^2}\big(\log r_{i_*}\mathds{1}_{d=3} + \mathds{1}_{d\geq 4}\big),
\end{eqnarray}
where in the last inequality we used Lemma~\ref{l:equilibrium-measure-cylinder}. 

\smallskip

By the definition of $I_0$, $I_1$ and $I_2$ in \eqref{eq:I0-I1-I2}, $|L_{i_*}|\sin\theta_{i_*}\geq \frac{\kappa h_r}{5(N+1)^2}$ and (a) $\theta_*\leq \frac\pi4$ if $i_*\in I_1$ or (b) $\theta_*>\frac\pi4$ and $r_{i_*}\geq \kappa \log^2 r$ if $i_*\in I_2$. Thus,  if $i_*\in I_1$ then 
\begin{eqnarray}\label{eq:upper-bound-hr}
I(\mu_0) - I(\mu_1)&\geq &c\frac{\kappa^2h_r^2}{r_{i_*}r^2}\big(\log r_{i_*}\mathds{1}_{d=3} + \mathds{1}_{d\geq 4}\big)
\geq c\frac{\kappa^2h_r^2}{r^3}\big(\log r\mathds{1}_{d=3} + \mathds{1}_{d\geq 4}\big)\nonumber\\
&= &c\kappa^2\big(\frac{\log^2 r}{r^2}\mathds{1}_{d=3} + \frac{1}{r^2}\mathds{1}_{d\geq 4}\big)
\geq c\kappa^2 \big(\mathrm{cap}\big(B(\mathcal L_0,\rho)\big)\big)^{-2}\\
&= &c\kappa^2 I(\mu_0)^2,\nonumber
\end{eqnarray}
where we used that $r_{i_*}\leq 2r$, \eqref{eq:capacity-asymptotics} and \eqref{eq:capacity-variational}; 
and if $i_*\in I_2$ then
\begin{eqnarray*}
I(\mu_0) - I(\mu_1)&\geq &c\frac{\kappa\log^2 r}{r^2}\big(\log\log r\mathds{1}_{d=3} + \mathds{1}_{d\geq 4}\big)\\
&\geq &c(\kappa\log\log r) \big(\mathrm{cap}\big(B(\mathcal L_0,\rho)\big)\big)^{-2}
= c(\kappa\log\log r) I(\mu_0)^2,
\end{eqnarray*}
where we used \eqref{eq:capacity-asymptotics} and \eqref{eq:capacity-variational}. The above two cases give \eqref{eq:BI-capacities-I}, and hence Lemma~\ref{l:upper-bound-3} is proven in the case $I_1\cup I_2\neq\emptyset$. 
\end{proof}

\begin{proof}[Proof of Lemma~\ref{l:upper-bound-3} in the case $I_1\cup I_2=\emptyset$]
The overall idea of the proof is very similar to the one in the first case. However, since the projection of any $L_i$, with $i\in I_0$, onto $Ox_1$-axis is now quite small, the corresponding subset of the cylinder $B([0,re_1],\rho)$ supports only a small fraction of the equilibrium measure, so that its push-forward onto $B(L_i,\rho)$ would not make a substantial increase of the energy. To overcome this, we modify the definition of the set $A_{i_*}$ from the previous proof by making it big enough to support a sufficient amount of the equilibrium measure. 

\smallskip

By the assumption of the lemma, there exists $1\leq i_*\leq N$ such that $|L_{i_*}|\sin\theta_{i_*}\geq \frac{\kappa h_r}{N+1}$. Fix such $i_*$. Since $|L_i|\sin\theta_i < \frac{\kappa h_r}{5(N+1)^2}$ for all $i\notin I_0$, there exists a subinterval $\widehat L_{i_*}$ of $L_{i_*}$ such that $|\widehat L_{i_*}| \geq \frac{\kappa h_r}{5(N+1)^2}$ and 
\begin{equation}\label{eq:distance-widehatL}
d(\widehat L_{i_*}, L_i) \geq \frac{\kappa h_r}{5(N+1)^2},\quad\text{for all }i\notin I_0.
\end{equation}
Note also that for any $i\in I_0$, the projection of $L_i$ on the $Ox_1$-axis has length at most $\kappa\log^2 r$. Moreover, $\kappa\geq 1$, since $I_1\cup I_2=\emptyset$, and $\kappa\leq \sqrt r$, since all the endpoints $o_i$ of the line segments $L_i$ are in the ball $B(r)$. Let 
\[
\Delta = \kappa\log^5 r.
\]
From the sequence of line segments 
\[
(L_{(1)},\ldots L_{(N+1)}) = \big(\widehat L_{i_*}, L_{i_*+1},\ldots, L_{N+1}, L_{i_*-1},L_{i_*-2},\ldots, L_1\big), 
\]
we select a family of their subsegments $L_1', \ldots, L_{N'}'$ with pairwise disjoint projections on the $Ox_1$-axis as follows. 
Let $\mathcal L_1' = L_{(1)}$. Let $B_1$ be the set of points in $\R^d$ whose projection on $Ox_1$-axis is at distance smaller than $4\rho+\Delta$ from the projection of $\mathcal L_1'$ and define $\mathcal L_2' = L_{(2)}\setminus B_1$. 
For $2\leq k\leq N$, let $B_k$ be the set of points in $\R^d$ whose projection on $Ox_1$-axis is at distance smaller than $4\rho$ from the projection of $\mathcal L_k'$ and define $\mathcal L_{k+1}' = L_{(k+1)}\setminus (B_1\cup\ldots\cup B_k)$. 
Note that $\{\mathcal L_1',\ldots, \mathcal L_{N+1}'\}$ is the set of at most $N+1$ line segments. 
From these line segments we keep $\mathcal L_1'$ $(=\widehat L_{i_*})$ and those segments that are subsegments of the segments $\{L_i, i\notin I_0\}$, and delete all the other ones. We order and denote the remaining line segments by $L_1', \ldots, L_{N'}'$ so that the projection of $L_i'$ on the $Ox_1$-axis is to the left of the projection of $L_{i+1}'$ for each $i$, and let $i_*'$ be such that $L_{i_*'}' = \widehat L_{i_*}$. This way, the distance between the projections of $L_{i-1}'$ and $L_i'$ is at least $4\rho$ when $i-1, i\neq i_*'$ and is at least $4\rho+\Delta$ when $i-1=i_*'$ or $i=i_*'$. Furthermore, 
\begin{equation}\label{eq:distance-widehatL-2}
d(L_{i_*'}', L_i')\geq \frac{\kappa h_r}{5(N+1)^2},\quad \text{for all } i\neq i_*',
\end{equation}
by \eqref{eq:distance-widehatL}. 

\smallskip

Let $r'$ be the sum of $2\Delta$ and the lengths of the projections of $L_1',\ldots, L_{N'}'$ on the $Ox_1$-axis. Note that $r' \geq r - 4\rho (N+1) - \kappa\log^2r(N+1)$, in particular (e.g.\ by \eqref{eq:capacity-asymptotics}), for some $C=C(d,\rho,N)$,
\[
\mathrm{cap}\big(B([0,re_1],\rho)\big) - \mathrm{cap}\big(B([0,r'e_1],\rho)\big)\leq C\kappa\log^2 r.
\]
Hence, by the monotonicity of the capacity, it suffices to prove that 
\[
\mathrm{cap}\Big(\bigcup_{i=1}^{N'}B(L_i',\rho)\Big)
- \mathrm{cap}\big(B([0,r'e_1],\rho)\big)
\geq c\kappa \log^3r'.
\]
\emph{By abuse of notation, we omit the primes and write in the following $L_i$ instead of $L_i'$, $N$ instead of $N'$, $r$ instead of $r'$, and $i_*$ instead of $i_*'$. (Note that in the new notation $L_{i_*}$ denotes what was previously denoted by $\widehat L_{i_*}$.)}

\smallskip

In the new notation, let $\mathcal L_0 = [0,re_1]$ and $\mathcal L_1 = \bigcup_{i=1}^NL_i$. 
Consider the points $a_1=0$, 
\begin{eqnarray*}
a_i &= &\big(\sum_{k=1}^{i-1} r_k, 0,\ldots,0\big),\text{ for }2\leq i\leq i_*,\\
a_i &= &\big(\sum_{k=1}^{i-1} r_k + 
2\Delta, 0,\ldots,0\big),\text{ for }i_*+1\leq i\leq N+1.
\end{eqnarray*}
and define the sets 
\[
A_i = \big\{x\in B(\mathcal L_0,\rho)\,:\, a_i(1)\leq x(1)<a_{i+1}(1)\big\},\,\text{ for }1\leq i\leq N.
\]
Note that $A_i$ are disjoint subsets of $B(\mathcal L_0,\rho)$ and by the subadditivity of the capacity, 
\[
\mathrm{cap}\big(B(\mathcal L_0,\rho)\big) - \mathrm{cap}\big(\bigcup_{i=1}^N A_i\big)\leq C.
\]
Hence it suffices to prove that 
\[
\mathrm{cap}\big(B(\mathcal L_1,\rho)\big) - \mathrm{cap}\big(\bigcup_{i=1}^N A_i\big)\geq c\kappa \log^3r.
\]

Let $\mu_0$ be the normalized equilibrium measure of $\bigcup_{i=1}^N A_i$. By \eqref{eq:capacity-variational}, it suffices to prove that for some probability measure $\mu_1$ on $B(\mathcal L_1,\rho)$, 
\[
\frac{1}{I(\mu_1)} - \frac{1}{I(\mu_0)} \geq c\kappa  \log^3r, 
\]
where $I(\mu) = \iint \|x-y\|^{2-d}\mu(dx)\mu(dy)$, for which it suffices to show that  
\begin{equation}\label{eq:BI-capacities-I-case2}
I(\mu_0) - I(\mu_1) \geq c\kappa\log^3 r I(\mu_0)^2.
\end{equation}
As in the case $I_1\cup I_2\neq\emptyset$, we identify $\mu_1$ as a push-forward of $\mu_0$ by some function $g: \bigcup_{i=1}^N A_i\to B(\mathcal L_1,\rho)$.
We write $\theta_i$ for the angle in $[0,\frac\pi2]$ between $L_i$ and $\R e_1$. Let $r_i = |L_i|\cos\theta_i$ be the length of the projection of $L_i$ onto the $Ox_1$-axis. By definition, $\sum_{i=1}^Nr_i +2\Delta= r$. 
Let $b_i$ and $b_i'$ be the end-vertices of the line segment $L_i$ so that $b_i(1)\leq b_i'(1)$. Fix some rotations $\psi_1, \ldots, \psi_N$ of $\R^d$ such that 
\[
\overrightarrow{b_ib_i'} = |L_i|\psi_i(e_1),
\]
define the function $g:\bigcup_{i=1}^N A_i\to B(\mathcal L_1,\rho)$ by
\begin{eqnarray*}
g(x) &= &b_i + \psi_i\Big(\big(\frac{|L_i|}{r_i}\big(x(1)-a_i(1)\big),x(2),\ldots, x(d)\big)\Big),\,\, \text{for } x\in A_i,\,\, i\neq i_*,\\
g(x) &= &b_{i_*} + \psi_{i_*}\Big(\big(\frac{|L_{i_*}|}{r_{i_*} +2\Delta}\big(x(1)-a_{i_*}(1)\big),x(2),\ldots, x(d)\big)\Big),\,\,\text{for } x\in A_{i_*},\\
\end{eqnarray*}
and consider the measure $\mu_1(\cdot) = \mu_0(g^{-1}(\cdot))$. It is a probability measure on $B(\mathcal L_1,\rho)$ and, by the change of variables formula, 
\[
I(\mu_1) = \iint \|x_1-y_1\|^{2-d}\mu_1(dx_1)\mu_1(dy_1) = \iint \|g(x)-g(y)\|^{2-d}\mu_0(dx)\mu_0(dy).
\]
Thus, 
\begin{eqnarray*}
I(\mu_0) - I(\mu_1) &= &\iint \frac{\|g(x)-g(y)\|^{d-2}-\|x-y\|^{d-2}}{\|x-y\|^{d-2}\|g(x)-g(y)\|^{d-2}}\mu_0(dx)\mu_0(dy)\\
&= &\sum_{i,j=1}^N\,\, \iint\limits_{A_i\times A_j} \frac{\|g(x)-g(y)\|^{d-2}-\|x-y\|^{d-2}}{\|x-y\|^{d-2}\|g(x)-g(y)\|^{d-2}}\mu_0(dx)\mu_0(dy).
\end{eqnarray*}
If $x,y\in A_{i_*}$, then
\begin{eqnarray*}
\|g(x)-g(y)\|^2 - \|x-y\|^2 &= &\frac{|L_{i_*}|^2}{(r_{i_*}+2\Delta)^2}|x(1)-y(1)|^2 - |x(1)-y(1)|^2\\
&\geq &c\,\frac{h_r^2}{\log^{10}r}\,|x(1)-y(1)|^2,
\end{eqnarray*}
and if $x,y\in A_i$, for $i\neq i_*$, then $\|g(x) - g(y)\|\geq \|x-y\|$. Thus, by a similar calculation, which gave \eqref{eq:Imu0-Imu1-case1} in the case $I_1\cup I_2\neq\emptyset$, we obtain that 
\begin{multline*}
\sum_{i=1}^N\,\, \iint\limits_{A_i\times A_i} \frac{\|g(x)-g(y)\|^{d-2}-\|x-y\|^{d-2}}{\|x-y\|^{d-2}\|g(x)-g(y)\|^{d-2}}\mu_0(dx)\mu_0(dy)\\
\geq \iint\limits_{A_{i_*}\times A_{i_*}} \frac{\|g(x)-g(y)\|^{d-2}-\|x-y\|^{d-2}}{\|x-y\|^{d-2}\|g(x)-g(y)\|^{d-2}}\mu_0(dx)\mu_0(dy)\\
\begin{aligned}
&\geq 
c\frac{1}{r}\big(\log\Delta \,\mathds{1}_{d=3} + \mathds{1}_{d\geq 4}\big)\mu_0(A_{i_*})
\stackrel{(a)}\geq
c\frac{\Delta}{r^2}\big(\log\Delta \,\mathds{1}_{d=3} + \mathds{1}_{d\geq 4}\big)\\
&\stackrel{(b)}\geq c\kappa\log^3 r\Big(\frac{\log^2r}{r^2}\mathds{1}_{d=3} + \frac{1}{r^2}\mathds{1}_{d\geq 4}\Big)
\stackrel{(c)}\geq c\kappa\log^3 r I(\mu_0)^2,
\end{aligned}
\end{multline*}
where in ($a$) we used that, by Lemma~\ref{l:equilibrium-measure-cylinder}, $\mu_0(A_{i_*}) \geq c\Delta/r$, in ($b$) the definition of $\Delta$, and in ($c$) \eqref{eq:capacity-asymptotics} and \eqref{eq:capacity-variational}.

\smallskip

In the case $I_1\cup I_2\neq\emptyset$, the function $g$ is a dilation, that is $\|g(x)-g(y)\|\geq \|x-y\|$ for all $x,y$. Thus, all the integrals over $A_i\times A_j$ are nonnegative. In the present case, the situation is different. One can still show, similarly to Lemma~\ref{l:gx-gy-case1}, that for all $i,j\neq i_*$, if $x\in A_i$ and $y\in A_j$, then $\|g(x)-g(y)\|\geq \|x-y\|$, hence the integral over $A_i\times A_j$ is nonnegative; however, if $x\in A_{i_*}$ and $y\in A_i$ for some $i\neq i_*$, then it is possible that $\|g(x)-g(y)\|<\|x-y\|$. We show that for all $i\neq i_*$,
\begin{equation}\label{eq:integral-cross-term}
\iint\limits_{A_{i_*}\times A_i} \frac{\|x-y\|^{d-2}-\|g(x)-g(y)\|^{d-2}}{\|x-y\|^{d-2}\|g(x)-g(y)\|^{d-2}}\mu_0(dx)\mu_0(dy)\leq C\kappa\,I(\mu_0)^2,
\end{equation}
which will complete the proof of \eqref{eq:BI-capacities-I-case2}. We use the following modification of Lemma~\ref{l:gx-gy-case1}. 
\begin{lemma}\label{l:gx-gy-case2}
For all $x\in A_{i_*}$ and $y\in A_i$ with $i\neq i_*$, 
\[
\|g(x) - g(y)\|\geq |\langle g(x)-g(y), e_1\rangle|\geq \|x-y\|-\Delta.
\]
\end{lemma}
\begin{proof}
Note that 
\[
\big|\langle g(a_i)-g(a_{i_*}),e_1\rangle\big| \geq 
\left\{\begin{array}{ll}
\|a_i-a_{i_*}\| + 4\rho - \Delta &\text{ if }i>i_*\\[4pt]
\|a_i-a_{i_*}\| + 4\rho + \Delta &\text{ if }i<i_*
\end{array}\right.
\]
and for $x=(x(1),0,\ldots, 0)\in A_i$, 
\[
\langle g(x)-g(a_i),e_1\rangle = 
\left\{\begin{array}{ll}
\|x-a_i\| &\text{ if }i\neq i_*\\[4pt]
\frac{r_{i_*}}{r_{i_*}+2\Delta}\|x- a_{i_*}\|\in\big[\|x-a_{i_*}\| - 2\Delta, \|x-a_{i_*}\|\big] &\text{ if }i=i_*.
\end{array}\right.
\]
The rest of the proof is the same as the proof of Lemma~\ref{l:gx-gy-case1} and we omit the details. 
\end{proof}

\smallskip

We decompose the integral in \eqref{eq:integral-cross-term} into the sum of the integrals over 
\begin{eqnarray*}
\widetilde A_1 &= &\{(x,y)\in A_{i_*}\times A_i\,:\,\|x-y\|\leq r^{4/5}\},\\
\widetilde A_2 &= &\{(x,y)\in A_{i_*}\times A_i\,:\,\|x-y\|> r^{4/5}\}.
\end{eqnarray*}
For $(x,y)\in \widetilde A_1$, by \eqref{eq:distance-widehatL-2} and Lemma~\ref{l:gx-gy-case2},
\begin{eqnarray*}
\|g(x) - g(y)\|^2 &\geq &\langle g(x)-g(y),e_1\rangle^2 + \Big(\frac{\kappa h_r}{5(N+1)^2}\Big)^2\\
&\geq &\big((\|x-y\| - \Delta)_+\big)^2 + c \kappa r \geq \|x-y\|^2,
\end{eqnarray*}
since $2\Delta \|x-y\| \ll \kappa r$. 
Hence, 
\[
\iint\limits_{\widetilde A_1} \frac{\|x-y\|^{d-2}-\|g(x)-g(y)\|^{d-2}}{\|x-y\|^{d-2}\|g(x)-g(y)\|^{d-2}}\mu_0(dx)\mu_0(dy)
\leq 0.
\]
For $(x,y)\in \widetilde A_2$, by Lemma~\ref{l:gx-gy-case2}, $\|g(x)-g(y)\|
\geq \|x-y\|-\Delta\geq \frac12 r^{4/5}$, for all large $r$, since $\kappa \leq \sqrt r$. Thus, 
\begin{multline*}
\iint\limits_{\widetilde A_2} \frac{\|x-y\|^{d-2}-\|g(x)-g(y)\|^{d-2}}{\|x-y\|^{d-2}\|g(x)-g(y)\|^{d-2}}\mu_0(dx)\mu_0(dy)\\
\leq \Big(\big(1+\frac{\Delta}{\frac12 r^{4/5}}\big)^{d-2} - 1\Big)\frac1{r^{4/5}}\mu_0(A_{i_*})\mu_0(A_i)
\stackrel{(a)}\leq C\frac{\Delta^2}{r^{13/5}}
\stackrel{(b)}\leq C\kappa\, I(\mu_0)^2,
\end{multline*}
where in (a) we used that $\mu_0(A_{i_*})\leq C\Delta/r$ (by Lemma~\ref{l:equilibrium-measure-cylinder}) and $\mu_0(A_i)\leq 1$, and in (b) that $\kappa\leq \sqrt r$, \eqref{eq:capacity-asymptotics} and \eqref{eq:capacity-variational}.
Putting the two estimates together gives \eqref{eq:integral-cross-term}, which completes the proof of Lemma~\ref{l:upper-bound-3} in the case $I_1\cup I_2=\emptyset$. 
\end{proof}

\begin{remark}\label{rem:hr-in-d3}
The only place in the proof of Lemma~\ref{l:upper-bound-3} for the Brownian interlacements, where the precise choice of $h_r$ ($=\sqrt{r\log r}$ for $d=3$ and $=\sqrt r$ for $d\geq 4$) is crucially used, is the final estimate \eqref{eq:upper-bound-hr} of the difference  $I(\mu_0)-I(\mu_1)$ in the case $I_1\neq \emptyset$. \end{remark}

\section{Proof of the lower bound}\label{sec:lower-bound}
In this section, we prove the lower bound in Theorem~\ref{thm:main}.

\smallskip

For $x\in \partial B(r)$ and $1\leq i\leq N$, consider the sets
\[
B^x_i = \Big\{z\in \R^d\,:\, 3^ih_r\leq d(z,\R x)\leq (3^i+1)h_r\text{ and } \langle z,x\rangle \in \big[\frac{i - 1/10}{N+1}, \frac{i+1/10}{N+1}\big]\Big\},
\]
and the corresponding sets of piecewise linear paths from $0$ to $x$
\[
\mathcal L^x = \Big\{\bigcup_{i=1}^{N+1} [o_{i-1}, o_i]\,:\, o_0=0, o_{N+1}=x \text{ and }o_i\in B^x_i,\,1\leq i\leq N\Big\}. 
\]
We also define 
\[
\mathcal L = \bigcup_{x\in \partial B(r)} \mathcal L^x. 
\]
For $L^x = \bigcup_{i=1}^{N+1}[x_{i-1},x_i]$ and $L^y = \bigcup_{i=1}^{N+1}[y_{i-1},y_i]$, define 
\[
d_H(L^x,L^y) = 
\max_{0\leq i\leq N}\big\{d(x_i,\ell^\infty_{y_i,y_{i+1}}), d(y_i, \ell^\infty_{x_i,x_{i+1}})\big\}\vee
\max_{1\leq i\leq N+1}\big\{d(x_i,\ell^\infty_{y_{i-1},y_i}), d(y_i, \ell^\infty_{x_{i-1},x_i})\big\}.
\]
Recall that $\mathcal V = \R^d\setminus \mathcal C$ is the vacant set of $\mathcal C$. 
\begin{lemma}\label{l:lower-bound-1}
There exist $c=c(d,N,\mathrm{law}(\mathcal C))>0$ and $C=C(d,N,\mathrm{law}(\mathcal C))<\infty$ such that for all $L\in \mathcal L$, 
\[
cf(r)\leq \mathsf P[L\subseteq \mathcal V]\leq Cf(r).
\]
\end{lemma}
\begin{lemma}\label{l:lower-bound-2}
There exist $c=c(d,N,\mathrm{law}(\mathcal C))>0$, $C=C(d,N,\mathrm{law}(\mathcal C))<\infty$ and $\epsilon=\epsilon(d,N)>0$ such that for all 
$L\in \mathcal L^{re_1}$ and $L'\in \mathcal L^v$, with $v\in \partial B(r)\cap B(re_1,\epsilon r)$, 
\[
\mathsf P\big[L'\subseteq \mathcal V\,|\,L\subseteq \mathcal V\big]\leq C\exp\Big(-\frac{c}{\delta_r}\min\big(d_H(L,L'), 1\big)\Big).
\]
\end{lemma}
We first show how the lemmas imply the lower bound in Theorem~\ref{thm:main} and then prove them for each of the three models separately. 
Let $S= \partial B(r)\cap B(re_1,\frac12\epsilon r)$, $S' = \partial B(r)\cap B(re_1,\epsilon r)$, and consider 
\[
X = \int_S\int_{B_1^x}\ldots\int_{B_N^x}\mathds{1}_{\mathrm{vis}(0,x_1,\ldots,x_N,x)}\sigma(dx)dx_1\ldots dx_N,
\]
where $\sigma(\cdot)$ is the surface measure on $\partial B(r)$. Note that $P_{\mathrm{vis}}^N(r)\geq \mathsf P(X>0)$. We use the second moment method to bound $\mathsf P(X>0)$. 
By Lemma~\ref{l:lower-bound-1}, 
\begin{eqnarray*}
\mathsf E[X] &=  &\int_S\int_{B_1^x}\ldots\int_{B_N^x}\mathsf P[\mathrm{vis}(0,x_1,\ldots,x_N,x)]\sigma(dx)dx_1\ldots dx_N\\
&\geq &cf(r)
\sigma(S)\prod_{i=1}^N\lambda_d(B_i^{re_1})
\geq c r^{d-1} (rh_r^{d-1})^N f(r).
\end{eqnarray*}
Furthermore, by rotation invariance, 
\begin{eqnarray*}
\mathsf E[X^2] &= &\int_{S^2}\int_{B_1^x\times\ldots\times B_N^x}\int_{B_1^y\times\ldots\times B_N^y}\mathsf P[\mathrm{vis}(0,x_1,\ldots,x_N,x)\cap\mathrm{vis}(0,y_1,\ldots,y_N,y)]\\ 
&&\qquad\qquad\qquad\qquad\qquad\qquad\qquad\qquad \sigma(dx)\sigma(dy)dx_1\ldots dx_Ndy_1\ldots dy_N\\ 
&\leq &C\mathsf E[X] \sup_{L\in\mathcal L^{re_1}}\int_{S'}\int_{B_1^x}\ldots\int_{B_N^x}\mathsf P[\mathrm{vis}(0,x_1,\ldots,x_N,x)\,|\,L\subseteq \mathcal V]\sigma(dx)dx_1\ldots dx_N.
\end{eqnarray*}
In order to apply Lemma~\ref{l:lower-bound-2}, we decompose the integral according to the value of $d_H(L^x,L)$, for $L^x = \bigcup_{i=1}^{N+1}[x_{i-1},x_i]$ ($x_0=0$, $x_{N+1}=x$) and $L = \bigcup_{i=1}^{N+1}[y_{i-1},y_i]$ ($y_0=0$, $y_{N+1}=re_1$). Note that if $d_H(L^x,L)=\delta$ then 
\begin{itemize}\itemsep4pt
\item
for all $1\leq i\leq N$, $x_i\in B(\ell^\infty_{y_{i-1},y_i},\delta)\cap B(\ell^\infty_{y_i,y_{i+1}},\delta)$,
\item
$x_{N+1}\in \partial B(r)\cap B(\ell^\infty_{y_N,y_{N+1}},\delta)$. 
\end{itemize}
By the definition of the domains $B^y_i$, there exist constants $c=c(d,N)$ and $C=C(d,N)$ such that the angles between the lines $\ell^\infty_{y_{i-1},y_i}$ and $\ell^\infty_{y_i,y_{i+1}}$, $1\leq i\leq N$, and between $\ell^\infty_{y_N,y_{N+1}}$ and $\ell^\infty_{0,y_{N+1}}$ all lie in $[ch_r/r,Ch_r/r]$. Thus, by Lemma~\ref{l:intersection-two-lines}, for $1\leq i\leq N$, $x_i$ lies in the cylinder of radius $\delta$ and length at most $C'\delta r/h_r$ centered at $y_i$ whose axis is aligned with $\ell^\infty_{y_{i-1},y_i}$, and $x_{N+1}\in\partial B(r)$ is within distance $C'\delta$ from $y_{N+1}$, for some constant $C' = C'(d,N)$. Thus, by decomposing the integral over the sets $\{(x, x_1,\ldots,x_N)\,: d_H(L^x,L)\in [k\delta_r, (k+1)\delta_r)\}$ and applying Lemma~\ref{l:lower-bound-2} and the above considerations, we obtain
\begin{multline*}
\mathsf E[X^2]\\
\leq C\mathsf E[X] \Big(
\sum_{k=0}^\infty e^{-ck} ((k+1)\delta_r)^{d-1}\,\big(((k+1)\delta_rr/h_r) ((k+1)\delta_r)^{d-1}\big)^N + Ce^{-c/\delta_r} r^{d-1}r^{dN}\Big)\\
\begin{aligned}
&= C\mathsf E[X]\Big(\sum_{k=0}^\infty e^{-ck} (k+1)^{(d-1)(N+1)+N}\,\delta_r^{(d-1)(N+1)}\beta_r^N + Ce^{-c/\delta_r} r^{d-1}r^{dN}\Big)\\
&\leq C \mathsf E[X]\delta_r^{(d-1)(N+1)}\beta_r^N.
\end{aligned}
\end{multline*}
Finally, by the Paley-Zygmund inequality, 
\[
\mathsf P(X>0) \geq \frac{\mathsf E[X]^2}{\mathsf E[X^2]} 
\geq 
c\frac{r^{d-1} (rh_r^{d-1})^N}{\delta_r^{(d-1)(N+1)}\beta_r^N} f(r) = c\Big(\frac{r}{\delta_r}\Big)^{d-1}\Big(\frac{h_r}{\delta_r}\Big)^{dN} f(r),
\]
as required. \qed

\subsection{Poisson-Boolean models}\label{sec:lower-bound-BM}
In this section, we prove Lemmas~\ref{l:lower-bound-1} and \ref{l:lower-bound-2} for the Poisson-Boolean models. 
\begin{proof}[Proof of Lemma~\ref{l:lower-bound-1}]
Let $L = \bigcup_{i=1}^{N+1}[x_{i-1},x_i]\in \mathcal L$. 
By \eqref{eq:BM-mu}, 
\[
\mathsf P^\alpha_\mathsf Q[L\subseteq \mathcal V] = \exp\big(-\alpha\mathsf E[\lambda_d(B(L,\rho))]\big)
\]
and 
\[
f(r) = \mathsf P^\alpha_\mathsf Q[[0,re_1]\subseteq \mathcal V] = 
\exp\big(-\alpha\mathsf E[\lambda_d(B([0,re_1],\rho))]\big).
\]
Thus, it suffices to prove that 
\[
\big|\mathsf E[\lambda_d(B(L,\rho))] - \mathsf E[\lambda_d(B([0,re_1],\rho))] \big| \leq C,
\]
for some $C=C(d,N,\mathsf Q)$. 
By the definition of $\mathcal L$ and $h_r$, 
\[
r\leq \sum_{i=1}^{N+1}\|x_{i-1}-x_i\| \leq r + C
\]
and 
\[
\lambda_d\big(B([x_{i-1},x_i],\rho)\cap B([x_i,x_{i+1}],\rho)\big) \leq C\rho^d, \quad 1\leq i\leq N,
\]
for some $C=C(d,N)$. 
Hence, the result follows from Lemma~\ref{l:volume-cylinder}.
\end{proof}

\begin{proof}[Proof of Lemma~\ref{l:lower-bound-2}]
Let $\epsilon=\frac15$. (The choice of $\epsilon$ does not matter in this proof.) Let $L= \bigcup_{i=1}^{N+1}[x_{i-1},x_i]$ and $L' = \bigcup_{i=1}^{N+1}[y_{i-1},y_i]$ be as in the lemma. 
By Lemma~\ref{l:lower-bound-1}, 
\[
C^{-1}\mathsf P^\alpha_\mathsf Q\big[L\subseteq \mathcal V\,|\,L'\subseteq \mathcal V\big]\leq  \mathsf P^\alpha_\mathsf Q\big[L'\subseteq \mathcal V\,|\,L\subseteq \mathcal V\big]\leq C\mathsf P^\alpha_\mathsf Q\big[L\subseteq \mathcal V\,|\,L'\subseteq \mathcal V\big].
\]
Therefore, without loss of generality, we may assume that 
$d_H(L,L') = d(y_i,\ell^\infty_{x_{i-1},x_i})$ or $d_H(L,L') = d(y_i,\ell^\infty_{x_i,x_{i+1}})$ for some $i$. We only consider the first case, the second is treated similarly. 
Fix $i_0$ such that $d_H(L,L') = d(y_{i_0},\ell^\infty_{x_{i_0-1},x_{i_0}})$.
By \eqref{eq:BM-mu}, 
\[
\mathsf P^\alpha_\mathsf Q\big[L'\subseteq \mathcal V\,|\,L\subseteq \mathcal V\big] = \exp\Big(-\alpha\big(\mathsf E\big[\lambda_d(B(L\cup L',\rho))\big] - \mathsf E\big[\lambda_d(B(L,\rho))\big]\big) \Big).
\]
Thus, it suffices to prove that for all $r\geq r_0 = r_0(d,N,\mathsf Q)$,
\[
\mathsf E\big[\lambda_d(B(L\cup L',\rho))\big] - \mathsf E\big[\lambda_d(B(L,\rho))\big] \geq \frac{c}{\delta_r}\min\big(d_H(L,L'), 1\big).
\]
Let $0<K_1<K_2$ be such that $\mathsf P[K_1\leq \rho\leq K_2]>0$.  Then 
\begin{eqnarray*}
\mathsf E\big[\lambda_d(B(L\cup L',\rho))\big] - \mathsf E\big[\lambda_d(B(L,\rho))\big]
&=
&\mathsf E\big[\lambda_d\big(B(L',\rho)\setminus B(L,\rho)\big)\big]\\
&\geq 
&\mathsf E\big[\lambda_d\big(B(L',\rho)\setminus B(L,\rho)\big);\,K_1\leq \rho\leq K_2\big].
\end{eqnarray*}
Let
\[
y = \frac{3y_{i_0-1} + 5y_{i_0}}{8}\quad \text{and} \quad 
y' = \frac{y_{i_0-1} + 7y_{i_0}}{8}.
\]
Note that for $\rho\leq K_2$ and all $r\geq r_0 = r_0(d,N,K_2)$, 
\[
B([y,y'],\rho)\cap B(L,\rho) = B([y,y'],\rho)\cap B([x_{i_0-1},x_{i_0}],\rho). 
\]
Thus, by Lemma~\ref{l:volume-difference}, 
\begin{eqnarray*}
\lambda_d(B(L',\rho)\setminus B(L,\rho))
&\geq &\lambda_d\big(B([y,y'],\rho)\setminus B([x_{i_0-1},x_{i_0}],\rho)\big)\\
&\geq &\lambda_d\big(B([y,y'],\rho)\setminus B(\ell^\infty_{x_{i_0-1},x_{i_0}},\rho)\big)\\
&\geq &c\rho^{d-2}\min(\rho,1)\|y-y'\|\min\big( d(y', \ell^\infty_{x_{i_0-1},x_{i_0}}), 1\big),
\end{eqnarray*}
for some $c=c(d)>0$. 

Since $d(y', \ell^\infty_{x_{i_0-1},x_{i_0}}) \geq \frac34d(y_{i_0},\ell^\infty_{x_{i_0-1},x_{i_0}}) = \frac34 d_H(L,L')$, $\|y-y'\|\geq \frac{r}{5(N+1)}$, and $\delta_r=\frac1r$, 
\[
\mathsf E\big[\lambda_d\big(B(L',\rho)\setminus B(L,\rho)\big)\big]\geq
c\min(K_1,1)^{d-1}\mathsf P[K_1\leq \rho\leq K_2]\frac{1}{\delta_r}\min\big(d_H(L,L'),1\big),
\]
for all $r\geq r_0 = r_0(d,N,K_2)$. The proof is completed. 
\end{proof}

\subsection{Poisson cylinders}\label{sec:lower-bound-PC}
In this section, we prove Lemmas~\ref{l:lower-bound-1} and \ref{l:lower-bound-2} for the Poisson cylinders. 
Recall that $K_\phi$ denotes the rotation of $K$ by $\phi$ and $\pi$ denotes the orthogonal projection on the hyperplane $H=\{z\in\R^d\,:\,z(1)=0\}$.

\begin{proof}[Proof of Lemma~\ref{l:lower-bound-1}]
Let $L = \bigcup_{i=1}^{N+1}[x_{i-1},x_i]\in \mathcal L$ and $L^0 = [0,x_{N+1}]$. 
By \eqref{eq:PC-mu-rho}, 
\[
\mathsf P^\alpha_\rho[L\subseteq \mathcal V] = \exp\Big(-\alpha
\int\limits_{SO_d}\lambda_{d-1}\big(B(\pi(L_\phi),\rho)\cap H\big)\,\nu(d\phi)\Big)
\]
and 
\[
f(r) = \mathsf P^\alpha_\rho[L^0\subseteq \mathcal V] = 
\exp\Big(-\alpha
\int\limits_{SO_d}\lambda_{d-1}\big(B(\pi(L^0_\phi),\rho)\cap H\big)\,\nu(d\phi)\Big).
\]
Hence it suffices to prove that
\[
\int\limits_{SO_d}\Big|\lambda_{d-1}\big(B(\pi(L_\phi),\rho)\cap H\big) - \lambda_{d-1}\big(B(\pi(L^0_\phi),\rho)\cap H\big)
\Big|\,\nu(d\phi)\leq C.
\]
Note that for any $\phi\in SO_d$, 
\[
\pi(L_\phi) = \bigcup_{i=1}^{N+1}[\pi(\phi(x_{i-1})),\pi(\phi(x_i))],\quad
\pi(L^0_\phi) = [0,\pi(\phi(x_{N+1}))],
\]
and 
\[
\max_{1\leq i\leq N} d\big(\pi(\phi(x_i)),\ell^\infty_{0,\pi(\phi(x_{N+1}))}\big)\leq 3^{N+1}h_r. 
\]
If $\phi\in SO_d$ satisfies $\|\pi(\phi(x_{N+1}))\|\geq h_r$, then by the definition of $\mathcal L$, 
\[
\|\pi(\phi(x_{N+1}))\|\leq \sum_{i=1}^{N+1}\|\pi(\phi(x_{i-1})) - \pi(\phi(x_i))\| \leq \|\pi(\phi(x_{N+1}))\| + \frac{Ch_r^2}{\|\pi(\phi(x_{N+1}))\|}
\]
and 
\[
\lambda_{d-1}\big(B([\pi(\phi(x_{i-1})),\pi(\phi(x_i))],\rho)\cap B([\pi(\phi(x_i)),\pi(\phi(x_{i+1}))],\rho)\big) \leq C\rho^{d-1}, \quad 1\leq i\leq N,
\]
for some $C=C(d,N)$. Hence, by Lemma~\ref{l:volume-cylinder}, 
\[
\big|\lambda_{d-1}\big(B(\pi(L_\phi),\rho)\cap H\big)
- \lambda_{d-1}\big(B(\pi(L^0_\phi),\rho)\cap H\big)\big|\leq \frac{Ch_r^2}{\|\pi(\phi(x_{N+1}))\|} = \frac{Cr}{\|\pi(\phi(x_{N+1}))\|},
\]
for some $C=C(d,N,\rho)$. 
If $\phi\in SO_d$ satisfies $\|\pi(\phi(x_{N+1}))\|\leq h_r$, then by Lemma~\ref{l:volume-cylinder},  
\[
\big|\lambda_{d-1}\big(B(\pi(L_\phi),\rho)\cap H\big)
- \lambda_{d-1}\big(B(\pi(L^0_\phi),\rho)\cap H\big)\big|\leq C h_r.
\]
Let 
\begin{eqnarray*}
\Phi_0 &= &\{\phi\in SO_d\,:\,\|\pi(\phi(x_{N+1}))\|\leq h_r\},\\ \Phi_k &= &\{\phi\in SO_d\,:\,\|\pi(\phi(x_{N+1}))\|\in [2^{k-1} h_r, 2^k h_r)\},\,\, k\geq 1.
\end{eqnarray*}
Note that 
$\nu(\Phi_k)\leq C(2^kh_r)^{d-1}/r^{d-1}$. 
Hence, if $k_* = \min\{k\,:\, 2^k h_r>r\}$, 
\begin{eqnarray*}
\Big|\lambda_{d-1}\big(B(\pi(L_\phi),\rho)\cap H\big) - \lambda_{d-1}\big(B(\pi(L^0_\phi),\rho)\cap H\big)
\Big|
&\leq &\frac{Ch_r h_r^{d-1}}{r^{d-1}}
+ \sum_{k=1}^{k_*}\frac{Cr}{2^{k-1}h_r}\frac{(2^kh_r)^{d-1}}{r^{d-1}}\\
&\leq &C\frac{r^{d/2}}{r^{d-1}} + C\frac{(2^{k_*}h_r)^{d-2}}{r^{d-2}}\leq C, 
\end{eqnarray*}
since $d\geq 3$. The proof is completed. 
\end{proof}

\begin{proof}[Proof of Lemma~\ref{l:lower-bound-2}]
Let $\epsilon = (100d(N+1))^{-1}$ and consider $L= \bigcup_{i=1}^{N+1}[x_{i-1},x_i]$ and $L' = \bigcup_{i=1}^{N+1}[y_{i-1},y_i]$ as in the lemma. 
By Lemma~\ref{l:lower-bound-1}, 
\[
C^{-1}\mathsf P^\alpha_\rho\big[L\subseteq \mathcal V\,|\,L'\subseteq \mathcal V\big]\leq  \mathsf P^\alpha_\rho\big[L'\subseteq \mathcal V\,|\,L\subseteq \mathcal V\big]\leq C\mathsf P^\alpha_\rho\big[L\subseteq \mathcal V\,|\,L'\subseteq \mathcal V\big].
\]
Therefore, without loss of generality, we may assume that 
$d_H(L,L') = d(y_i,\ell^\infty_{x_{i-1},x_i})$ or $d_H(L,L') = d(y_i,\ell^\infty_{x_i,x_{i+1}})$ for some $i$. We only consider the first case, the second is treated similarly. 
Fix $i_0$ such that $d_H(L,L') = d(y_{i_0},\ell^\infty_{x_{i_0-1},x_{i_0}})$.
By \eqref{eq:PC-mu-rho}, 
\[
\mathsf P^\alpha_\rho\big[L'\subseteq \mathcal V\,|\,L\subseteq \mathcal V\big] = \exp\Big(-\alpha\Big[
\lambda_{d-1}\big(B(\pi(L_\phi\cup L_\phi'),\rho)\cap H\big) - \lambda_{d-1}\big(B(\pi(L_\phi),\rho)\cap H\big)\Big]
\Big).
\]
Thus, it suffices to prove that 
\[
\lambda_{d-1}\big(B(\pi(L_\phi\cup L_\phi'),\rho)\cap H\big) - \lambda_{d-1}\big(B(\pi(L_\phi),\rho)\cap H\big) \geq \frac{c}{\delta_r}\min\big(d_H(L,L'), 1\big).
\]
Let
\[
y = \frac{3y_{i_0-1} + 5y_{i_0}}{8}\quad \text{and} \quad 
y' = \frac{y_{i_0-1} + 7y_{i_0}}{8}
\]
and note that $d(y',\ell^\infty_{x_{i_0-1},x_{i_0}})\geq \frac34 d(y_{i_0},\ell^\infty_{x_{i_0-1},x_{i_0}}) = \frac34d_H(L,L')$.
Thus, it suffices to prove that 
\[
\lambda_{d-1}\big(B(\pi(L_\phi\cup L_\phi'),\rho)\cap H\big) - \lambda_{d-1}\big(B(\pi(L_\phi),\rho)\cap H\big) \geq \frac{c}{\delta_r}\min\big(d(y',\ell^\infty_{x_{i_0-1},x_{i_0}}), 1\big).
\]

Let $H_s$ be the hyperplane $\{z \in\R^d\,:\,z(s)=0\}$ and $\pi_s$ the orthogonal projection on $H_s$. For $K\subseteq H_s$, denote by $B_s(K,\rho)$ the $(d-1)$-dimensional $\rho$-neighborhood of $K$ in $H_s$, $B_s(K,\rho) = B(K,\rho)\cap H_s$.
By the invariance of the Haar measure, we have 
\begin{multline*}
\lambda_{d-1}\big(B(\pi(L_\phi\cup L_\phi'),\rho)\cap H\big) - \lambda_{d-1}\big(B(\pi(L_\phi),\rho)\cap H\big)\\
\begin{aligned}
&=\,\, \int\limits_{SO_d} \lambda_{d-1}\big(B_1(\pi(L_\phi'),\rho)\setminus B_1(\pi(L_\phi),\rho)\big) \nu(d\phi)\\
&=\,\, \frac1d\sum_{s=1}^d \int\limits_{SO_d} \lambda_{d-1}\big(B_s(\pi_s(L_\phi'),\rho)\setminus B_s(\pi_s(L_\phi),\rho)\big) \nu(d\phi).
\end{aligned}
\end{multline*}
Fix $\phi\in SO_d$. 
Note that 
\[
\pi_s(L_\phi) = \bigcup_{i=1}^{N+1}[\pi_s(\phi(x_{i-1})),\pi_s(\phi(x_i))], \quad
\pi_s(L_\phi') = \bigcup_{i=1}^{N+1}[\pi_s(\phi(y_{i-1})),\pi_s(\phi(y_i))].
\]

By Lemma~\ref{l:vector-projections} applied to the vectors 
$u = x_{i_0} - x_{i_0-1}$ and $v=y' - x_{i_0-1}$, 
there exists $s_0$ such that 
\[
\|\pi_{s_0}(\phi(u))\|\geq \frac1{\sqrt d}\|u\|\quad\text{and}\quad
d\big(\pi_{s_0}(\phi(y')), \ell^\infty_{\pi_{s_0}(\phi(x_{i_0-1})),\pi_{s_0}(\phi(x_{i_0}))}\big)\geq \frac{1}{d} d(y', \ell^\infty_{x_{i_0-1},x_{i_0}}).
\]
From the first inequality and the facts that $\|x_{i_0-1}-x_{i_0}\|\geq cr$ and the angle between $\ell^\infty_{x_{i_0-1},x_{i_0}}$ and $\ell^\infty_{0,x_{N+1}}$ is at most $Ch_r/r$, it follows that 
\[
\|\pi_{s_0}(\phi(x_{N+1}))\|\geq \frac1{\sqrt d}r - Ch_r\geq \frac1{2\sqrt{d}}r.
\]
Furthermore, $\|\pi_{s_0}(\phi(x_{N+1})) - \pi_{s_0}(\phi(y_{N+1}))\|\leq \epsilon r$. Hence, by the definition of $\mathcal L$ and the choice of $s_0$ and $\epsilon$, 
\begin{multline*}
B([\pi_{s_0}(\phi(y)),\pi_{s_0}(\phi(y'))],\rho)\cap B(\pi_{s_0}(L_\phi),\rho)\\ 
= B([\pi_{s_0}(\phi(y)),\pi_{s_0}(\phi(y'))],\rho)\cap B([\pi_{s_0}(\phi(x_{i_0-1})),\pi_{s_0}(\phi(x_{i_0}))],\rho).
\end{multline*}
Therefore, 
\begin{multline*}
\lambda_{d-1}\big(B(\pi(L_\phi\cup L_\phi'),\rho)\cap H\big) - \lambda_{d-1}\big(B(\pi(L_\phi),\rho)\cap H\big)\\
\begin{aligned}
&\geq \frac1d \int\limits_{SO_d} \lambda_{d-1}\big(B_{s_0}([\pi_{s_0}(\phi(y)),\pi_{s_0}(\phi(y'))],\rho)\setminus B_{s_0}([\pi_{s_0}(\phi(x_{i_0-1})),\pi_{s_0}(\phi(x_{i_0}))],\rho)\big) \nu(d\phi)\\
&\geq \frac1d \int\limits_{SO_d} \lambda_{d-1}\big(B_{s_0}([\pi_{s_0}(\phi(y)),\pi_{s_0}(\phi(y'))],\rho)\setminus B_{s_0}(\ell^{\infty}_{\pi_{s_0}(\phi(x_{i_0-1})),\pi_{s_0}(\phi(x_{i_0}))},\rho)\big) \nu(d\phi).
\end{aligned}
\end{multline*}
By Lemma~\ref{l:volume-difference}, 
\begin{multline*}
\lambda_{d-1}\big((B([\pi_{s_0}(\phi(y)),\pi_{s_0}(\phi(y'))],\rho)\setminus B(\ell^{\infty}_{\pi_{s_0}(\phi(x_{i_0-1})),\pi_{s_0}(\phi(x_{i_0}))},\rho))\cap H_{s_0}\big)\\
\geq c\|\pi_{s_0}(\phi(y))-\pi_{s_0}(\phi(y'))\|\,\min\big(d\big(\pi_{s_0}(\phi(y')),\ell^\infty_{\pi_{s_0}(\phi(x_{i_0-1})),\pi_{s_0}(\phi(x_{i_0}))}\big), 1\big).
\end{multline*}
By the choice of $s_0$, the minimim on the right hand side is bounded from below by $\frac1{d}\min\big(d(y', \ell^\infty_{x_{i_0-1},x_{i_0}}),1\big)$. 
Furthermore, by the definition of $\mathcal L$ and the choice of $s_0$ and $\epsilon$, $\|\pi_{s_0}(\phi(y))-\pi_{s_0}(\phi(y'))\|\geq cr$. The result follows, since $\delta_r=\frac1r$. 
\end{proof}

\subsection{Brownian interlacements}\label{sec:lower-bound-BI}
In this section, we prove Lemmas~\ref{l:lower-bound-1} and \ref{l:lower-bound-2} for the Brownian interlacements. 
\begin{proof}[Proof of Lemma~\ref{l:lower-bound-1}]
Let $L = \bigcup_{i=1}^{N+1}[x_{i-1},x_i]\in \mathcal L$ with $x_0=0$ and $x_{N+1}=re_1$. By \eqref{eq:BI-capacity-rho}, 
\[
\mathsf P^\alpha_\rho[L\subseteq \mathcal V] = \exp\big(-\alpha\,\mathrm{cap}(B(L,\rho))\big)
\]
and 
\[
f(r) = \mathsf P^\alpha_\rho[[0,re_1]\subseteq \mathcal V] = 
\exp\big(-\alpha\,\mathrm{cap}(B([0,re_1],\rho))\big).
\]
Thus, it suffices to prove that 
\[
\big|\mathrm{cap}(B(L,\rho)) - \mathrm{cap}(B([0,re_1],\rho)) \big| \leq C,
\]
for some $C=C(d,N,\rho)$. By Lemma~\ref{l:upper-bound-3},  
$\mathrm{cap}(B(L,\rho)) \geq \mathrm{cap}(B([0,re_1],\rho)) - C$. Hence, it remains to prove that 
\[
\mathrm{cap}(B(L,\rho)) \leq \mathrm{cap}(B([0,re_1],\rho)) + C.
\]
We use a similar, but simpler, argument as in the proof of Lemma~\ref{l:upper-bound-3}. 
Let $L_i = [x_{i-1},x_i]$. By the definition of $\mathcal L$ and $h_r$, there exists $C=C(d,N)$ such that 
\begin{itemize}\itemsep4pt
\item
$\sum_{i=1}^{N+1}|L_i|\leq r + C(1+\log r\mathds{1}_{d=3})$, 
\item
$B(L_i,\rho)\cap B(L_{i+1},\rho)\subseteq B(x_i,C\rho)$ for all $1\leq i\leq N$. 
\end{itemize}
Thus, we can choose $\widetilde L_i = [u_i,v_i]\subseteq L_i$ such that $u_0=0$, $v_{N+1} = re_1$, and for all $i$, 
\begin{itemize}\itemsep4pt
\item
$B(\widetilde L_i,\rho)\cap B(\widetilde L_{i+1},\rho) = \emptyset$, 
\item
$\|v_i - u_{i+1}\|\leq D$ for some $D=D(d,N,\rho)$. 
\end{itemize}
Let $R= \sum_{i=1}^{N+1}|\widetilde L_i| + (D+4\rho)N$. Note that $R\leq r + (D+4\rho)N + C(1+\log r\mathds{1}_{d=3})$. 
By the subadditivity of capacity and Lemma~\ref{l:capacity-cylinder-d3}, for some $C=C(d,N,\rho)$, 
\[
\mathrm{cap}(B(L,\rho))\leq \mathrm{cap}\big(\bigcup_{i=1}^{N+1}B(\widetilde L_i,\rho)\big) + C
\]
and
\[
\mathrm{cap}(B([0,Re_1],\rho))\leq \mathrm{cap}(B([0,re_1],\rho)) + C.
\]
Thus, it suffices to prove that 
\[
\mathrm{cap}\big(\bigcup_{i=1}^{N+1}B(\widetilde L_i,\rho)\big)\leq \mathrm{cap}(B([0,Re_1],\rho)).
\]
Let 
\[
a_i =\Big(\sum_{j=1}^{i-1}|\widetilde L_j| + (D+4\rho)(i-1),0,\ldots,0\Big) \quad 
b_i = a_i + |\widetilde L_i|e_1,\quad 1\leq i\leq N+1,
\]
and define $A_i = B([a_i,b_i],\rho)$. Note that $A_i\subseteq B([0,Re_1],\rho)$ and $A_i\cap A_j=\emptyset$ for all $i\neq j$. 
Consider the function $g: \bigcup_{i=1}^{N+1}B(\widetilde L_i,\rho)\to \bigcup_{i=1}^{N+1}A_i$ such that the restriction of $g$ to $B(\widetilde L_i,\rho)$ is a rotation satisfying $g(u_i)= a_i$ and $g(v_i) = b_i$; in particular, $g(B(\widetilde L_i,\rho)) = A_i$ for all $i$. Note that 
\begin{itemize}\itemsep4pt
\item
for $x,y\in B(\widetilde L_i,\rho)$, $\|g(x)-g(y)\|=\|x-y\|$, 
\item
for $x\in \widetilde L_i$ and $y\in \widetilde L_j$ with $i<j$, 
\begin{eqnarray*}
\|x-y\|&\leq &\|x-v_i\| + \sum_{k=i+1}^{j-1} \|u_k - v_k\| + \|u_j - y\| + D(j-i)\\ 
&= &\|g(x) - b_i\| + \sum_{k=i+1}^{j-1} \|a_k - b_k\| + \|a_j - g(y) \| + D(j-i)\\ 
&= &\|g(x)-g(y)\| - 4\rho(j-i),
\end{eqnarray*}
\item
for $x\in B(\widetilde L_i,\rho)$ and $y\in B(\widetilde L_j,\rho)$ with $i<j$, by writing $x'$ for the closest point to $x$ in $\widetilde L_i$ and $y'$ for the closest point to $y$ in $\widetilde L_j$, 
\[
\|x-y\|\leq \|x'-y'\| + 2\rho\leq \|g(x')-g(y')\| - 2\rho \leq \|g(x) - g(y)\|.
\]
\end{itemize}
All in all, $g$ is a dilation. Therefore, using the variational definition of the capacity in the same way as in the proof of Lemma~\ref{l:upper-bound-3}, we obtain that 
\[
\mathrm{cap}\big(\bigcup_{i=1}^{N+1}B(\widetilde L_i,\rho)\big)\leq 
\mathrm{cap}\big(g\big(\bigcup_{i=1}^{N+1}B(\widetilde L_i,\rho)\big)\big) = \mathrm{cap}\big(\bigcup_{i=1}^{N+1} A_i\big) \leq \mathrm{cap}(B([0,Re_1],\rho)).
\]
The proof is completed. 
\end{proof}

\begin{proof}[Proof of Lemma~\ref{l:lower-bound-2}]
Let $\epsilon=\epsilon(d)>0$ be small enough, to be specified later. Let $L= \bigcup_{i=1}^{N+1}[x_{i-1},x_i]$ and $L' = \bigcup_{i=1}^{N+1}[y_{i-1},y_i]$ be as in the lemma. 
By Lemma~\ref{l:lower-bound-1}, 
\[
C^{-1}\mathsf P^\alpha_\rho\big[L\subseteq \mathcal V\,|\,L'\subseteq \mathcal V\big]\leq  \mathsf P^\alpha_\rho\big[L'\subseteq \mathcal V\,|\,L\subseteq \mathcal V\big]\leq C\mathsf P^\alpha_\rho\big[L\subseteq \mathcal V\,|\,L'\subseteq \mathcal V\big].
\]
Therefore, without loss of generality, we may assume that 
$d_H(L,L') = d(y_i,\ell^\infty_{x_{i-1},x_i})$ or $d_H(L,L') = d(y_i,\ell^\infty_{x_i,x_{i+1}})$ for some $i$. We only consider the first case, the second is treated similarly. 
Fix $i_0$ such that $d_H(L,L') = d(y_{i_0},\ell^\infty_{x_{i_0-1},x_{i_0}})$.
By definition, 
\begin{eqnarray*}
\mathsf P^\alpha_\rho\big[L'\subseteq \mathcal V\,|\,L\subseteq \mathcal V\big] &= &\exp\Big(-\alpha\big(\mathrm{cap}(B(L\cup L',\rho))- \mathrm{cap}(B(L,\rho))\big) \Big),\\
\mathsf P^\alpha_\rho\big[L\subseteq \mathcal V\,|\,L'\subseteq \mathcal V\big] &= &\exp\Big(-\alpha\big(\mathrm{cap}(B(L\cup L',\rho))- \mathrm{cap}(B(L',\rho))\big) \Big).
\end{eqnarray*}
Thus, it suffices to prove that 
\[
\mathrm{cap}(B(L\cup L',\rho))- \min\big(\mathrm{cap}(B(L,\rho)),\mathrm{cap}(B(L',\rho))\big) \geq \frac{c}{\delta_r}\min\big(d_H(L,L'), 1\big).
\]
We will use the following observation, which is immediate from the definition of the capacity \eqref{def:equilibrium-measure-and-capacity} and the strong Markov property: For any compact sets $A, B$ and $S\subseteq B\setminus A$, for all $R$ large enough, 
\[
\mathrm{cap}(A\cup B) - \mathrm{cap}(A) 
= \mathsf P_{\mu_R}[H_B<\infty, H_A=\infty]
\geq \mathsf P_{\mu_R}[H_S<H_A]\,\inf_{z\in S}\mathsf P_z[H_A=\infty].
\]
In the following, we will identify $S$ as a subset of $B(L',\rho)\setminus B(L,\rho)$ or of $B(L,\rho)\setminus B(L',\rho)$. 
Let
\[
x = \frac{3x_{i_0-1} + 5x_{i_0}}{8},\,\,
x' = \frac{x_{i_0-1} + 7x_{i_0}}{8},\,\,
y = \frac{3y_{i_0-1} + 5y_{i_0}}{8},\,\,
y' = \frac{y_{i_0-1} + 7y_{i_0}}{8}
\]
and note that $d([x,x'],[y,y'])\geq \frac14d_H(L,L')$.
Let 
\[
\varphi(r) = \left\{\begin{array}{ll} \frac{r}{\log r} & d=3\\ r & d\geq 4\end{array}\right.
\]
We claim that for all $R$ large enough, either 
\begin{equation}\label{eq:BI-lower-bound-hitting}
\mathsf P_{\mu_R}\big[H_{B([x,x'],2\rho)}<H_{B(L',2\rho)}\big]\geq c\varphi(r)\,\,\,\text{or}\,\,\,
\mathsf P_{\mu_R}\big[H_{B([y,y'],2\rho)}<H_{B(L,2\rho)}\big]\geq c\varphi(r).
\end{equation}
By the assumption of the lemma, $B(L,2\rho)\cup B(L',2\rho)\subseteq B([0,re_1],2\epsilon r)$. Let 
\begin{eqnarray*}
U &= &\Big\{z\in \partial B([0,re_1],2\epsilon r)\,:\,\big|z(1)-\frac{i_0-1/4}{N+1}r\big|\leq \epsilon r\Big\},\\
A &= &\Big\{z\in B([x,x'],2\rho)\cup B([y,y'],2\rho)\,:\,\big|z(1)-\frac{i_0-1/4}{N+1}r\big|\leq \epsilon r\Big\},\\
B &= &\big(B(L,2\rho)\cup B(L',2\rho)\big)\setminus \big(B([x,x'],2\rho)\cup B([y,y'],2\rho)\big).
\end{eqnarray*}
For all $R$ large enough, 
\begin{multline*}
\mathsf P_{\mu_R} \big[H_{B([x,x'],2\rho)\cup B([y,y'],2\rho)}<H_B\big]\\
\geq 
\mathsf P_{\mu_R} \big[H_U=H_{B([0,re_1],2\epsilon r)}<\infty\big]
\inf_{z\in U}\mathsf P_z\big[H_{B([x,x'],2\rho)\cup B([y,y'],2\rho)}<H_B\big].
\end{multline*}
By Lemma~\ref{l:equilibrium-measure-cylinder}, 
\begin{eqnarray*}
\mathsf P_{\mu_R} \big[H_U=H_{B([0,re_1],2\epsilon r)}<\infty\big]
&\geq &c\mathsf P_{\mu_R} \big[H_{B([0,re_1],2\epsilon r)}<\infty\big]\\
&= &c\,\mathrm{cap}\big(B([0,re_1],2\epsilon r)\big)
\geq cr^{d-2}. 
\end{eqnarray*}
By \eqref{eq:hitting-probability-capacity}, for each $z\in U$, 
\begin{eqnarray*}
\mathsf P_z\big[H_{B([x,x'],2\rho)\cup B([y,y'],2\rho)}<H_B\big]
&\geq &\mathsf P_z\big[H_A<H_B\big]
\geq \mathsf P_z\big[H_A<\infty\big] - \sup_{z'\in B}\mathsf P_{z'}\big[H_A<\infty\big]\\
&\geq &c (\epsilon r)^{2-d}\mathrm{cap}(A) - Cr^{2-d}\mathrm{cap}(A)
\stackrel{(*)}\geq c r^{2-d}\mathrm{cap}(A)\\
&\geq &cr^{2-d}\varphi(r),
\end{eqnarray*}
where ($*$) holds if $\epsilon=\epsilon(d)>0$ is small enough, and the last inequality follows from \eqref{eq:capacity-asymptotics}.
Since the boundaries of $B([x,x'],2\rho)$ and $B([y,y'],2\rho)$ intersect on a set of measure $0$, it follows that either 
\[
\mathsf P_{\mu_R}\big[H_{B([x,x'],2\rho)}<H_{B\cup B([y,y'],2\rho)}\big]
\geq \tfrac12 c\varphi(r)\,\,\,\text{or}\,\,\,
\mathsf P_{\mu_R}\big[H_{B([y,y'],2\rho)}<H_{B\cup B([x,x'],2\rho)}\big]
\geq \tfrac12 c\varphi(r).
\]
Since $B(L,2\rho)\subseteq B\cup B([x,x'],2\rho)$ and $B(L',2\rho)\subseteq B\cup B([y,y'],2\rho)$, \eqref{eq:BI-lower-bound-hitting} follows. 

\smallskip

Let
\begin{eqnarray*}
S &= &\big\{z\in B([x,x'],\rho)\,:\, d(z, B(L',\rho))\geq \tfrac12d([x,x'],[y,y'])\big\},\\
S' &= &\big\{z\in B([y,y'],\rho)\,:\, d(z, B(L,\rho))\geq \tfrac12d([x,x'],[y,y'])\big\}.
\end{eqnarray*}
Note that there exists $c=c(d,\rho)>0$ such that 
\[
\inf_{z\in \partial B([x,x'],2\rho)\setminus B(L',2\rho)} \mathsf P_z\big[H_S<H_{B(L',\rho)}\big],\,\,
\inf_{z\in \partial B([y,y'],2\rho)\setminus B(L,2\rho)} \mathsf P_z\big[H_{S'}<H_{B(L,\rho)}\big] \geq c.
\]
Thus, either
\[
\mathsf P_{\mu_R}\big[H_S<H_{B(L',\rho)}\big] \geq c\varphi(r)\,\,\,\text{or}\,\,\,
\mathsf P_{\mu_R}\big[H_{S'}<H_{B(L,\rho)}\big] \geq c\varphi(r).
\]
Let
\[
\psi(r) = \left\{\begin{array}{ll} \frac{1}{\log r} & d=3\\ 1 & d\geq 4\end{array}\right.
\]
and note that $\varphi(r)\psi(r) = \frac1{\delta_r}$. Thus, in order to complete the proof of the lemma, it suffices to show that 
\begin{equation}\label{eq:escape-probability-cylinder}
\inf_{z\in S}\mathsf P_z\big[H_{B(L',\rho)}=\infty\big],\,\, \inf_{z\in S'}\mathsf P_z\big[H_{B(L,\rho)}=\infty\big] \geq c\psi(r)\min\big(d_H(L,L'),1\big).
\end{equation}
We only prove the first inequality, the second one follows by symmetry.

\smallskip

Let $T_1$ be the first exit time of the Brownian motion from the infinite cylinder $B(\ell_{y,y'}^\infty,r^{4/5})$ and $T_2$ the first time the Brownian motion is at distance $r/10(N+1)$ from $S$. Let $G = \{v\in \partial B(\ell_{y,y'}^\infty,r^{4/5})\,:\,d(v,S)\leq r/10(N+1)\}$. Since $d(S,B([y_{i-1},y_i],\rho))>r/10(N+1)$ for all $i\neq i_0$, we obtain that for each $z\in S$, 
\begin{equation}\label{eq:escape-probability-cylinder-1}
\mathsf P_z[H_{B(L',\rho)} = \infty]\geq \mathsf P_z\big[T_1< H_{B(\ell_{y,y'}^\infty,\rho)}\wedge T_2\big]\,\inf_{v\in G}\mathsf P_v[H_{B(L',\rho)} = \infty].
\end{equation}
By \eqref{eq:BM-hitting} and the fact that the projection of the Brownian motion on the hyperplane orthogonal to $\ell_{y,y'}^\infty$ is a $(d-1)$-dimensional Brownian motion, 
\[
\mathsf P_z[T_1<H_{B(\ell_{y,y'}^\infty,\rho)}] \geq c\psi(r)\min\big(d(z,B(\ell_{y,y'}^\infty,\rho)),1\big).
\]
Furthermore, if $z'$ is the closest to $z$ point on $\ell_{y,y'}^\infty$, then 
\[
\mathsf P_z\big[T_2<H_{B(\ell_{y,y'}^\infty,\rho)}\wedge T_1\big]
\leq \mathsf P_z\big[H_{B(z',2\rho)^c}<H_{B(z',\rho)}\big]\,\sup_{z''\in B(S,4\rho)}\mathsf P_{z''}[T_2< T_1].
\]
By \eqref{eq:BM-hitting}, the first probability is at most $C\min\big(d(z,B(\ell_{y,y'}^\infty,\rho)),1\big)$. The second probability is bounded from above by $Ce^{-cr^{1/5}}$, since for any $v\in B(\ell_{y,y'}^\infty,r^{4/5})$, the Brownian motion started in $v$ will exit $B(\ell_{y,y'}^\infty,r^{4/5})$ before leaving the ball $B(v,4r^{4/5})$ with probability at least $c=c(d)>0$, implying that the Brownian motion leaves $B(S, r/10(N+1))$ before leaving $B(\ell_{y,y'}^\infty,r^{4/5})$ with probability at most $(1-c)^{c' r^{1/5}}$, for some $c' = c'(d,N)$. 
Thus, 
\[
\mathsf P_z\big[T_1< H_{B(\ell_{y,y'}^\infty,\rho)}\wedge T_2\big]
\geq c\psi(r)\min\big(d(z,B(\ell_{y,y'}^\infty,\rho)),1\big)\geq c\psi(r)\min\big(d_H(L,L'),1\big).
\]
Next, by the definition of $\mathcal L$ and $h_r$, $B(L',\rho)\subseteq B(\ell_{y,y'}^{\infty},r^{3/5})$; also $B(L',\rho)\subseteq B(y,2r)$. Thus, for any $v\in \partial B(\ell_{y,y'}^\infty, r^{4/5})$, 
\[
\mathsf P_v [H_{B(L',\rho)}=\infty] \geq \mathsf P_v\big[H_{\partial B(\ell_{y,y'}^\infty,3r)}<H_{B(\ell_{y,y'}^\infty,r^{3/5})}\big]\,\inf_{v'\in \partial B(\ell_{y,y'}^\infty,3r)}\mathsf P_{v'}[H_{B(y,2r)}=\infty].
\]
The first probability is at least $c=c(d)>0$ by \eqref{eq:BM-hitting} and the fact that the projection of the Brownian motion on the hyperplane orthogonal to $\ell_{y,y'}^\infty$ is a $(d-1)$-dimensional Brownian motion. The second probability is bounded from below by $c=c(d)>0$ by \eqref{eq:BM-escape}. 
Plugging the two bounds in \eqref{eq:escape-probability-cylinder-1} gives \eqref{eq:escape-probability-cylinder}. The proof of Lemma~\ref{l:lower-bound-2} is completed. 
\end{proof}

\section*{Acknowledgements}
The research of both authors has been supported by the DFG Priority Program 2265 ``Random Geometric Systems'' (Project number 443849139).

\end{document}